\documentclass[preprint,3p]{elsarticle}

\journal{Journal of Computational Physics}
\makeatletter
\AddToHook{cmd/ps@pprintTitle/after}{%
  \g@addto@macro\@oddfoot{\unskip}%
  \let\@evenfoot\@oddfoot
}
\makeatother

\usepackage{amsmath}
\usepackage{amssymb}
\usepackage{amsthm}
\usepackage{mathtools} 
\usepackage{mathrsfs}  
\usepackage{bm}        

\usepackage{booktabs}
\usepackage{caption}
\usepackage{subcaption}
\usepackage{acronym}
\usepackage[section]{placeins}

\usepackage[
    colorlinks=true,
    linkcolor=blue,
    citecolor=blue,
    urlcolor=blue,
    pdftitle={A Discontinuous Galerkin Method for the Intrinsic Beam Model with Kelvin--Voigt Damping},
    pdfauthor={Shivam Sundriyal, Christian Bleffert, Lukas Dreyer, Melven Roehrig-Zoellner, Gregor Gassner, Vadym Aizinger}
]{hyperref}

\newtheoremstyle{jcpstyle}
  {5pt plus 1pt minus 1pt} 
  {5pt plus 1pt minus 1pt} 
  {\itshape}               
  {}                       
  {\bfseries}              
  {}                       
  {.5em}                   
  {}                       

\theoremstyle{jcpstyle}
\newcounter{globaltheorem}[section]
\renewcommand{\theglobaltheorem}{\thesection.\arabic{globaltheorem}}

\newtheorem{theorem}[globaltheorem]{Theorem}

\newtheorem{corollary}[globaltheorem]{Corollary}

\newtheorem{problem}[globaltheorem]{Problem}
\newtheorem{remark}[globaltheorem]{Remark}

\newcommand{\llbracket}{\left[\!\left[}
\newcommand{\rrbracket}{\right]\!\right]}
\newcommand{\jump}[1]{\llbracket #1 \rrbracket}
\newcommand{\avg}[1]{\ensuremath{\bigl\{\!\!\bigl\{ #1 \bigr\}\!\!\bigr\}}}

\newcommand{\inproduct}[3][\Omega]{\left( #2 , #3 \right)_{#1}}
\newcommand{\engnorm}[1]{\left\| #1 \right\|_{\bm{\Gamma}} }
\newcommand{\dengnorm}[1]{\left\| #1 \right\|_{\bm{\Gamma},k} }

\newcommand{\bndeval}[2]{\bigl[#1^{\top} #2\bigr]_{0}^{\ell}}
\newcommand{\bndevallong}[2]{\bigl[(#1)^{\top} #2\bigr]_{0}^{\ell}}
\newcommand{\externalboundary}[2]{\left[ \big(#1\big)^{\top} #2 \right]_0^\ell}
\newcommand{\externalboundarysingle}[2]{\left[ #1^{\top} #2 \right]_0^\ell}
\newcommand{\interfacepoint}[3][x_{i+1/2}]{\left( \big(#2\big)^{\top} #3 \right)\Big|_{#1}}
\newcommand{\interfacepointsingle}[3][x_{i+1/2}]{\left( \big(#2\big)^{\top} #3 \right)\Big|_{#1}}

\newcommand{\Amat}{\bm{A}}
\newcommand{\Bmat}{\bm{B}}
\newcommand{\dampmat}{\bm{C}_\tau}
\newcommand{\flexmat}{\bm{C}}
\newcommand{\Emat}{\bm{E}}
\newcommand{\Gammamat}{\bm{\Gamma}}
\newcommand{\Hmat}{\bm{H}}
\newcommand{\massmat}{\bm{M}}
\newcommand{\Pimat}{\bm{\Pi}}
\newcommand{\Psimat}{\bm{\Psi}}
\newcommand{\qext}{\bm{q}_\text{ext}}
\newcommand{\rtau}{\bm r_\tau}
\newcommand{\rtauex}{\bm r_{\tau,\mathrm{ex}}}
\newcommand{\rtauh}{\bm r_{\tau h}}

\newcommand{\uh}{\bm u_h}
\newcommand{\uone}{\bm{u}_1}
\newcommand{\uoneh}{\bm u_{1h}}
\newcommand{\utwo}{\bm{u}_2}
\newcommand{\utwoh}{\bm u_{2h}}
\newcommand{\zerovec}[1]{\mathbf{0}_{#1}}
\newcommand{\zeromat}[2]{\mathbf{0}_{#1,#2}}

\newcommand{\linformA}[2]{\bm{\mathcal{A}}\bigl(#1, #2\bigr)}
\newcommand{\linformB}[2]{\bm{\mathcal{B}}\bigl(#1, #2\bigr)}
\newcommand{\linformC}[2]{\bm{\mathcal{C}}\bigl(#1, #2\bigr)}
\newcommand{\dlinformA}[2]{\bm{\mathcal{A}}_h\left(#1, #2\right)}
\newcommand{\dlinformB}[2]{\bm{\mathcal{B}}_h\left(#1, #2\right)}
\newcommand{\dlinformC}[2]{\bm{\mathcal{C}}_h\left(#1, #2\right)}

\newcommand{\R}{\mathbb{R}}
\newcommand{\Lone}{\mathcal{L}_1}
\newcommand{\Ltwo}{\mathcal{L}_2}
\newcommand{\idmat}[2]{\mathbf{I}_{#1, #2}}

\newcommand{\ibvp}{IBVP}
\newcommand{\trixi}{\texttt{Trixi.jl}}
\begin{document}

\hypersetup{pageanchor=false}
\begin{frontmatter}

\title{A Discontinuous Galerkin Method for the Intrinsic Beam Model with Kelvin--Voigt Damping}

\author[bayreuth]{Shivam Sundriyal\corref{cor1}}
\ead{Shivam.Sundriyal@uni-bayreuth.de}
\ead[url]{https://orcid.org/0009-0004-3959-1976}
\cortext[cor1]{Corresponding author}
\author[dlr]{Christian Bleffert}
\author[hannover]{Lukas Dreyer}
\author[dlr]{Melven Röhrig-Zöllner}
\author[cologne]{Gregor Gassner}
\author[bayreuth]{Vadym Aizinger}

\affiliation[bayreuth]{organization={Chair of Scientific Computing, University of Bayreuth},
            addressline={Universitätsstraße 30}, 
            city={Bayreuth},
            postcode={95447}, 
            country={Germany}}

\affiliation[dlr]{organization={German Aerospace Center (DLR), Institute of Software Technology},
            addressline={Linder Höhe}, 
            city={Cologne},
            postcode={51147}, 
            country={Germany}}

\affiliation[hannover]{organization={Leibniz University Hannover},
            addressline={Welfengarten 1}, 
            city={Hannover},
            postcode={30167}, 
            country={Germany}}

\affiliation[cologne]{organization={University of Cologne},
            addressline={Albertus-Magnus-Platz}, 
            city={Cologne},
            postcode={50923}, 
            country={Germany}}

\begin{abstract}
We present an energy-stable discontinuous Galerkin (DG) spatial discretization of the intrinsic beam equations with Kelvin--Voigt damping. We reformulate the governing equations by splitting the total sectional resultants into elastic and viscous parts, thereby exposing a mixed system without mixed space--time derivatives. For homogeneous cantilever boundary data, the continuous system obeys an energy dissipation identity. With characteristic upwinding and alternating auxiliary traces the DG discretization is energy stable at the semidiscrete level for the stated homogeneous split boundary closure. The method is implemented in the Julia-based simulation framework \trixi. A physically compatible manufactured solution with all state and viscous components active assesses convergence, while a rotating-beam test verifies the analytic constant-spin branch and the actuated energy balance.
\end{abstract}

\begin{highlights}
\item Energy-stable DG scheme for Kelvin--Voigt damped intrinsic beams
\item Auxiliary variable removes mixed space--time derivatives
\item Semi-discrete energy estimate proved for homogeneous split boundary data
\item A compatible manufactured solution assesses high-order convergence
\item A rotating-beam test verifies the steady branch and actuated energy balance
\end{highlights}

\begin{keyword}
Discontinuous Galerkin \sep Kelvin--Voigt Damping \sep Geometrically Exact Beam Theory \sep Rotating beams
\end{keyword}

\end{frontmatter}
\hypersetup{pageanchor=true}

\section{Introduction}
\label{sec:introduction}

Highly flexible helicopter and wind-turbine blades can undergo large deflections,
rotations, and vibrations under complex dynamic loads. Fully three-dimensional
finite element models are often very expensive for repeated system-level
simulations, hence motivating representation of the blades as one-dimensional beams. The
classical linear Euler--Bernoulli~\cite{Truesdell1960,Timoshenko1953} and
Timoshenko models~\cite{Timoshenko1921} do not describe the large motions of
such structures.

To address these limitations, the geometrically exact intrinsic beam theory formulated by
Hodges~\cite{Hodges1990,Hodges2003} accommodates large deformation, twist, and
anisotropic material behavior. The term ``intrinsic'' refers to the use of
strain, curvature, velocity, and sectional-resultant variables rather than
displacement and rotation variables as primary unknowns. This avoids the
additional nonlinearities associated with displacement--rotation formulations
while retaining the geometric nonlinearity required to describe large motions.
In the remainder of this paper, we refer to this framework as the \emph{intrinsic
beam model} or the \emph{intrinsic beam equations}.

Accurate representation of damping is important for predicting vibration,
fatigue, and aeroelastic response. Physical damping may arise from internal
material mechanisms, joints and interfaces, and aerodynamic interactions
\cite{Lazan1968,BanksInman1991,Leishman2006,Johnson2013}. Numerical damping
introduced by a time integrator serves a different purpose and is not a
constitutive model of these mechanisms. In this work, we consider rate-dependent
internal material damping described by a generalized Kelvin--Voigt law. Artola
et al.~\cite{Artola2021} incorporated this law into a
geometrically exact intrinsic beam model and analyzed the resulting
dissipation and equilibria.

Although the continuous damped model has been formulated and analyzed, its
numerical solution remains challenging. Numerical treatments of the intrinsic
equations include an energy-consistent modal scheme for the undamped
equations~\cite{PatilAlthoff2011}, variable-order finite elements for the
nonlinear fully intrinsic equations~\cite{PatilHodges2011}, and nonlinear
normal-mode expansions~\cite{Palacios2011}. The Kelvin--Voigt term introduces
mixed space--time derivatives and a mixed hyperbolic--parabolic character;
structurally, the damped system is a nonlinear, twelve-component analogue of
the strongly damped wave equation, whose finite element error analysis is
classical in the linear scalar case~\cite{LarssonThomeeWahlbin1991}. DG
treatments of related linear scalar models include an LDG method with
alternating fluxes for diffusive--viscous wave equations~\cite{LingShuYan2023}
and an HDG method for strongly damped wave problems~\cite{GuptaYadav2025}.
Those analyses do not address the nonlinear twelve-component intrinsic beam
structure, the constitutive damping-resultant split, or the characteristic beam
boundary closure considered here. DG methods are natural candidates because of
their high-order accuracy and element locality~\cite{CockburnKarniadakisShu2000}.
For the \emph{undamped} intrinsic beam equations, an energy-stable DG scheme was
constructed in~\cite{Bleffert2025}, and boundary feedback stabilization was
established at the continuous level
in~\cite{RodriguezLeugering2020,Rodriguez2022}; to the authors' knowledge,
however, no DG method with a semi-discrete energy analysis has been proposed
for the Kelvin--Voigt damped intrinsic beam model. Thus, the novelty claimed
here is the damping-resultant split and its matching DG energy identity, not
the first high-order weak discretization of intrinsic beam equations.

In this paper, we develop an energy-stable discontinuous Galerkin treatment of the damped intrinsic beam model.
Our first contribution is a reformulation of the governing equations that removes the mixed space--time derivatives introduced by the Kelvin--Voigt law. We introduce an auxiliary variable representing the rate-dependent contribution to the sectional force and moment resultants. The resulting first-order-in-time mixed system separates the hyperbolic intrinsic-beam structure from the damping contribution and is therefore more amenable to energy analysis and DG discretization. This structure allows us to extend energy-stable DG ideas from the undamped intrinsic beam equations~\cite{Bleffert2025} to the damped setting.

Our second contribution is an energy analysis of the reformulated continuous
problem. Assuming that the sectional mass and flexibility matrices as well as the
effective Kelvin--Voigt operator are uniformly symmetric positive definite, we
derive a mechanical energy--dissipation identity that retains the boundary-work
term. The mechanical energy is the sum of the kinetic and elastic energies defined in
Eq.~\eqref{eq:physical_energy}. For homogeneous cantilever boundary data and
square-integrable external loads, the identity yields an \textit{a priori} stability
bound.

Our third contribution is a DG spatial discretization that mimics the continuous
energy structure. We reuse key ideas from the undamped case~\cite{Bleffert2025}
for the hyperbolic part. Characteristic upwinding supplies non-negative jump and
boundary dissipation, while compatible auxiliary traces like alternating
LDG and central BR1 traces cancel the remaining interior coupling. Under the
stated homogeneous split cantilever closure, the resulting method is energy
stable at the semidiscrete level. The method is implemented in the Julia-based
simulation framework \trixi, and the nodal formulation used in the numerical
experiments is detailed in~\ref{app:dgsem}. The experiments assess
convergence and the influence of Kelvin--Voigt damping on the solution behavior.

The remainder of this paper is organized as follows. Section~\ref{sec:damped_intrinsic_beam_equations} presents the governing equations of the damped intrinsic beam model and their reformulation for DG discretization. Section~\ref{sec:continuous_analysis} derives the continuous boundary-work identity and the stability bound for homogeneous cantilever data. Section~\ref{sec:discretization} details the DG formulation, including numerical flux construction and a semidiscrete stability proof. Section~\ref{sec:numerical_analysis} presents numerical experiments assessing convergence and dissipative behavior. Finally, Section~\ref{sec:conclusion} summarizes our findings, and Section~\ref{sec:future_work} outlines directions for future work.

\section{Damped Intrinsic Beam Equations}
\label{sec:damped_intrinsic_beam_equations}

The intrinsic beam model was first formulated by Hodges~\cite{Hodges1990}, with the fully intrinsic form developed in~\cite{Hodges2003}.
It models the dynamics of initially curved and twisted beams of anisotropic material.
The underlying beam is idealized by a one-dimensional reference line (the elastic axis).
It is geometrically exact, meaning it incorporates nonlinearities that enable the representation of large motions, including significant displacements of the reference line and substantial deformations of the beam's cross-sections.

Artola et al.~\cite{Artola2021} incorporated Kelvin--Voigt damping into this framework by augmenting the constitutive laws used to close the model in~\cite{Hodges2003} with strain-rate dependent terms.
In the following, we first summarize the intrinsic beam model and present the constitutive laws closing the formulation and incorporating the Kelvin--Voigt damping.
We adapt the notation established in~\cite{Palacios2011}.

Let $x\in[0,\ell]$ denote the arc-length coordinate along the undeformed beam axis, where $\ell\in\R_+$ is the beam length, and let $t\in\R_+$ denote time. Let $\bm{f}=\bm{f}(x,t)\in\R^3$ and $\bm{m}=\bm{m}(x,t)\in\R^3$ be the force and moment sectional resultants, and let $\bm{f}_e=\bm{f}_e(x,t)\in\R^6$ be the vector of distributed external forces and moments per unit length. Moreover, let $\bm{p}=\bm{p}(x,t)\in\R^3$ and $\bm{h}=\bm{h}(x,t)\in\R^3$ denote the generalized linear and angular momenta, let $\bm{v}=\bm{v}(x,t)\in\R^3$ and $\bm{\omega}=\bm{\omega}(x,t)\in\R^3$ denote the linear and angular inertial velocities, and let $\bm{\gamma}=\bm{\gamma}(x,t)\in\R^3$ and $\bm{\kappa}=\bm{\kappa}(x,t)\in\R^3$ denote the generalized strains and curvatures. Then, the governing equations of the intrinsic beam model read
\begin{align}
\label{eq:intrinsic_beam_original_1}
    \begin{bmatrix}\dot{\bm{p}}\\ \dot{\bm{h}}\end{bmatrix}
    -
    \begin{bmatrix}\bm{f}'\\ \bm{m}'\end{bmatrix}
    -\Emat\begin{bmatrix}\bm{f}\\ \bm{m}\end{bmatrix}
    +\Lone\!\left(\begin{bmatrix}\bm{v}\\ \bm{\omega}\end{bmatrix}\right)\begin{bmatrix}\bm{p}\\ \bm{h}\end{bmatrix}
    +\Ltwo\!\left(\begin{bmatrix}\bm{f}\\ \bm{m}\end{bmatrix}\right)\begin{bmatrix}\bm{\gamma}\\ \bm{\kappa}\end{bmatrix}
    =\bm{f}_e,
    \\[1.5ex]
\label{eq:intrinsic_beam_original_2}
    \begin{bmatrix}\dot{\bm{\gamma}}\\ \dot{\bm{\kappa}}\end{bmatrix}
    -
    \begin{bmatrix}\bm{v}'\\ \bm{\omega}'\end{bmatrix}
    +\Emat^\top\begin{bmatrix}\bm{v}\\ \bm{\omega}\end{bmatrix}
    -\Lone^\top\!\left(\begin{bmatrix}\bm{v}\\ \bm{\omega}\end{bmatrix}\right)\begin{bmatrix}\bm{\gamma}\\ \bm{\kappa}\end{bmatrix}
    =\zerovec{6}.
\end{align}
Partial derivatives with respect to space and time are denoted by primes \((\cdot)'\) and dots \(\dot{(\cdot)}\), respectively.

The matrix $\Emat=\Emat(x)$ is defined by
\begin{equation}
    \Emat :=
    \begin{bmatrix}
        \tilde{\bm{\kappa}_0} & \zeromat{3}{3} \\
        \tilde{\bm{e}_1}      & \tilde{\bm{\kappa}_0}
    \end{bmatrix},
\end{equation}
where $\bm{\kappa}_0=\bm{\kappa}_0(x)\in\R^3$ denotes the initial curvature and twist of the beam and
$\bm{e}_1=[1,0,0]^\top\in\R^3$ is the first Cartesian basis vector. The operator
$\tilde{(\cdot)}:\R^3\to\R^{3\times 3}$ maps $\bm{a}=(a_1,a_2,a_3)^\top$ to the skew-symmetric matrix
\begin{equation}
    \tilde{\bm{a}}=
    \begin{bmatrix}
        0   &-a_3   &a_2  \\
        a_3 &0      &-a_1 \\
        -a_2&a_1    &0
    \end{bmatrix},
    \qquad
    \tilde{\bm{a}}\bm{b}=\bm{a}\times\bm{b}\quad \text{for all } \bm{b}\in\R^3.
\end{equation}

The linear operators $\Lone,\Ltwo:\R^6\to\R^{6\times 6}$ simplify the formulation of the governing equations and are defined, for $\bm{c}_1,\bm{c}_2\in\R^3$, by
\begin{align}
\Lone\!\left(\begin{bmatrix}\bm{c}_1\\ \bm{c}_2\end{bmatrix}\right)
&=
\begin{bmatrix}
\tilde{\bm{c}_2} &\zeromat{3}{3}\\
\tilde{\bm{c}_1} &\tilde{\bm{c}_2}
\end{bmatrix},
\qquad
\Ltwo\!\left(\begin{bmatrix}\bm{c}_1\\ \bm{c}_2\end{bmatrix}\right)
=
\begin{bmatrix}
\zeromat{3}{3} &\tilde{\bm{c}_1}\\
\tilde{\bm{c}_1} &\tilde{\bm{c}_2}
\end{bmatrix}.
\end{align}
They satisfy
\begin{align}
\label{eq:l_operator_properties_1}
\Lone^\top\!\left(\begin{bmatrix}\bm{c}_1\\ \bm{c}_2\end{bmatrix}\right)
\begin{bmatrix}\bm{c}_1\\ \bm{c}_2\end{bmatrix}
&=\zerovec{6},
\\[0.5ex]
\label{eq:l_operator_properties_2}
\Lone\!\left(\begin{bmatrix}\bm{b}_1\\ \bm{b}_2\end{bmatrix}\right)
\begin{bmatrix}\bm{c}_1\\ \bm{c}_2\end{bmatrix}
&=
\Ltwo^\top\!\left(\begin{bmatrix}\bm{c}_1\\ \bm{c}_2\end{bmatrix}\right)
\begin{bmatrix}\bm{b}_1\\ \bm{b}_2\end{bmatrix}
\end{align}
which follow from skew-symmetry of the cross-product matrix $\tilde{\bm{a}}^\top=-\tilde{\bm{a}}$ and $\tilde{\bm{c}}_1\bm{c}_2=-\tilde{\bm{c}}_2\bm{c}_1$.

Equations~\eqref{eq:intrinsic_beam_original_1}--\eqref{eq:intrinsic_beam_original_2} form two coupled $6$-component PDEs (equivalently, four $3$-component PDEs) for the eight vector-valued unknowns
$\bm{f},\bm{m},\bm{v},\bm{\omega},\bm{\gamma},\bm{\kappa},\bm{p},\bm{h}$
($12$ scalar equations for $24$ scalar unknowns).
The system is closed by the following two constitutive laws ($12$ additional scalar relations): one relating sectional forces and moments to strains and curvatures, and one relating sectional velocities to generalized momenta
\begin{align}
\label{eq:constitutive_law_forces}
    \begin{bmatrix}
        \bm{f} \\ \bm{m}
    \end{bmatrix}
    &=
    \flexmat^{-1}
    \begin{bmatrix}
        \bm{\gamma} \\ \bm{\kappa}
    \end{bmatrix}
    +
    \dampmat\flexmat^{-1}
    \begin{bmatrix}
        \dot{\bm{\gamma}} \\ \dot{\bm{\kappa}}
    \end{bmatrix},
    \\[1.5ex]
\label{eq:constitutive_law_velocities}
    \begin{bmatrix}
        \bm{v} \\ \bm{\omega}
    \end{bmatrix}
    &=
    \massmat^{-1}
    \begin{bmatrix}
        \bm{p} \\ \bm{h}
    \end{bmatrix}.
\end{align}
Here, $\flexmat=\flexmat(x)\in\R^{6\times 6}$ is the flexibility matrix and
$\massmat=\massmat(x)\in\R^{6\times 6}$ is the mass matrix, determined by the
material properties and cross-sectional geometry. We restrict attention to the
nondegenerate setting in which both matrices are uniformly symmetric positive
definite. For $\flexmat$, this ensures that the inverse used
in~\eqref{eq:constitutive_law_forces} exists and that the elastic quadratic form
induces a norm; singular compliance matrices describe constrained idealizations,
such as perfectly inextensible or unshearable limits, which are not considered
here. The matrix $\dampmat=\dampmat(x)$ is the Kelvin--Voigt
damping-coefficient matrix. We impose no separate symmetry assumption on
$\dampmat$; instead, the symmetry and positivity assumption is placed on the
effective Kelvin--Voigt operator $\dampmat\flexmat^{-1}$, as stated precisely
below. Setting $\dampmat=\zeromat{6}{6}$ recovers the undamped original
equations~\cite{Hodges2003}.
We assume $\massmat$, $\flexmat$, and $\dampmat$ are sufficiently smooth functions of $x$.
The common block structures of $\massmat$ and $\flexmat$, together with the
assumption on the effective damping operator, are summarized in \ref{app:matrices}
(see also~\cite{Rodriguez2022,Hodges2003}).

For compactness, we introduce the primary variables $\uone,\utwo:[0,\ell]\times\R_+\to\R^6$,
\begin{equation}
    \uone := \begin{bmatrix}\bm{v}\\ \bm{\omega}\end{bmatrix},
    \qquad
    \utwo := \flexmat^{-1}\begin{bmatrix}\bm{\gamma}\\ \bm{\kappa}\end{bmatrix}.
\end{equation}
Following~\cite{Artola2021}, substituting the constitutive laws~\eqref{eq:constitutive_law_forces}--\eqref{eq:constitutive_law_velocities}
into the intrinsic beam equations~\eqref{eq:intrinsic_beam_original_1}--\eqref{eq:intrinsic_beam_original_2} yields the damped intrinsic beam system
\begin{align}
\label{eq:damped_equation_1}
    \massmat\dot{\uone}
    -(\utwo+\dampmat\dot{\utwo})'
    -\Emat(\utwo+\dampmat\dot{\utwo})
    +\Lone(\uone)\massmat\uone
    +\Ltwo(\utwo+\dampmat\dot{\utwo})\,\flexmat\utwo
    &= \bm{f}_e,
    \\[1.5ex]
\label{eq:damped_equation_2}
    \flexmat\dot{\utwo}
    -\uone'
    +\Emat^\top\uone
    -\Lone^\top(\uone)\,\flexmat\utwo
    &= \zerovec{6}.
\end{align}
The Kelvin--Voigt term introduces mixed space--time derivatives through
\(
(\dampmat\dot{\utwo})',
\)
so that~\eqref{eq:damped_equation_1}--\eqref{eq:damped_equation_2} is second order in the mixed sense (first order in time, but containing $\partial_x\partial_t$ terms).

Directly discretizing such mixed derivatives is inconvenient and complicates stability analysis. We therefore first derive an equivalent formulation without mixed derivatives, resulting in a first-order-in-time and second-order-in-space system with a hyperbolic operator coupled to a dissipative (parabolic) operator. Concretely, solving~\eqref{eq:damped_equation_2} for $\dot{\utwo}$ and substituting into~\eqref{eq:damped_equation_1} yields
\begin{equation}
\label{eq:advection_diffusion_eqn}
    \Gammamat\,\dot{\bm{u}}
    =
    \mathcal{D}\,\bm{u}''
    +
    \mathcal{A}(\bm{u})\,\bm{u}'
    +
    \mathcal{S}(\bm{u})\,\bm{u}
    +
    \begin{bmatrix}
         \bm{f}_e\\ \zerovec{6}
    \end{bmatrix},
    \qquad
    \bm{u}:=\begin{bmatrix}\uone\\ \utwo\end{bmatrix}\in\R^{12},
\end{equation}
with $\Gammamat=\operatorname{diag}(\massmat,\flexmat)$ as introduced below.

Here, $\mathcal{D}$ collects the Kelvin--Voigt dissipation, $\mathcal{A}(\bm{u})$ contains the first-order spatial (propagation) terms, and $\mathcal{S}(\bm{u})$ collects zero-order contributions. Explicit expressions for $\mathcal{D}$, $\mathcal{A}$, and $\mathcal{S}$ are given in \ref{app:nonlinear_advection_diffusion}, stated there for $x$-independent coefficient matrices; spatially varying $\Emat$, $\flexmat$, $\dampmat$ generate additional lower-order terms.

This viewpoint connects the problem to a broad DG literature on hyperbolic--parabolic systems~\cite{Cockburn1999,Cockburn2001,CockburnShu1998a}. However, the nonlinearity in $\mathcal{A}(\bm{u})$ complicates the construction of stable numerical fluxes and consistent boundary conditions~\cite{CockburnLinShu1989,CockburnShu1989}, motivating a further reformulation that is more suitable for analysis and computation.

To eliminate the mixed space--time derivatives, we introduce the auxiliary variable
$\rtau=\rtau(x,t)\in\R^6$ by
\begin{equation}
\label{eq:def_rtau}
    \rtau := \dampmat\,\dot{\utwo}.
\end{equation}
By construction, $\rtau$ represents the Kelvin--Voigt (rate-dependent) contribution to the internal force/moment resultants. Substituting $\dampmat\dot{\utwo}=\rtau$ into the damped intrinsic beam system~\eqref{eq:damped_equation_1}--\eqref{eq:damped_equation_2} yields the equivalent first-order-in-time system

\begin{align}
\label{eq:first_order_system_1}
   & \massmat\dot{\uone}
    -\utwo'
    -\rtau'
    -\Emat\utwo
    -\Emat\rtau
    +\Lone(\uone)\massmat\uone
    +\Ltwo(\utwo)\,\flexmat\utwo
    +\Ltwo(\rtau)\,\flexmat\utwo
    =
    \bm{f}_e,
    \\[1.5ex]
\label{eq:first_order_system_2}
    &\flexmat\dot{\utwo}
    -\uone'
    +\Emat^\top\uone
    -\Lone^\top(\uone)\,\flexmat\utwo
    =
    \zerovec{6},
    \\[1.5ex]
\label{eq:first_order_system_3}
&    \flexmat\dampmat^{-1}\rtau
    =
    \uone'
    -\Emat^\top\uone
    +\Lone^\top(\uone)\,\flexmat\utwo.
\end{align}
Equation~\eqref{eq:first_order_system_3} is obtained by combining~\eqref{eq:def_rtau} with~\eqref{eq:first_order_system_2} and expresses $\rtau$
in terms of $\uone$ and $\utwo$ without introducing mixed derivatives.

To write~\eqref{eq:first_order_system_1}--\eqref{eq:first_order_system_3} compactly, define
$\Gammamat=\Gammamat(x)\in\R^{12\times 12}$, $\Pimat\in\R^{12\times 12}$, and $\Hmat=\Hmat(x)\in\R^{6\times 6}$ by
\begin{align}
    \Gammamat
    \coloneqq
    \begin{bmatrix}
        \massmat        &\zeromat{6}{6}       \\
        \zeromat{6}{6}  &\flexmat
    \end{bmatrix},
    \qquad
    \Pimat
    \coloneqq
    \begin{bmatrix}
        \zeromat{6}{6} &-\idmat{6}{6} \\
        -\idmat{6}{6}  &\zeromat{6}{6}
    \end{bmatrix},
    \qquad
    \Hmat
    \coloneqq
    \dampmat\,\flexmat^{-1}.
\end{align}
Here, $\idmat{6}{6}$ denotes the $6\times 6$ identity matrix. We assume that
$\Hmat=\dampmat\,\flexmat^{-1}$ is uniformly symmetric positive definite. This
is the condition required for coercive material dissipation and makes the
inverse auxiliary relation~\eqref{eq:first_order_system_3} well defined. Since
$\Hmat$ and $\flexmat$ are invertible, $\dampmat=\Hmat\flexmat$ is invertible
but need not be symmetric; no commutation assumption is required. All damped
configurations in Section~\ref{sec:numerical_analysis} satisfy this hypothesis.

Define the primary unknown $\bm{u}=\bm{u}(x,t)\in\R^{12}$ and the zero-padded auxiliary vector $\bm{r}=\bm{r}(x,t)\in\R^{12}$ by
\begin{equation}
    \bm{u} := \begin{bmatrix} \uone \\ \utwo \end{bmatrix},
    \qquad
    \bm{r} := \begin{bmatrix} \zerovec{6} \\ \rtau \end{bmatrix}.
\end{equation}
This notation exposes the physical energy ledger:
\begin{equation}
\label{eq:physical_energy}
E(t)=\frac12\int_0^\ell
\left(\uone^\top\massmat\uone+\utwo^\top\flexmat\utwo\right)\,\mathrm{d}x.
\end{equation}
The two terms are kinetic and elastic energy, respectively; $\utwo+\rtau$ is
the total sectional resultant, $\rtau^\top\Hmat^{-1}\rtau$ is the material
dissipation density, and $\uone^\top(\utwo+\rtau)$ is the boundary-power flux.

Collecting all zero-order terms in~\eqref{eq:first_order_system_1}--\eqref{eq:first_order_system_2}, we introduce
$\bm{Q}(\bm{u},\bm{r})\in\R^{12}$ via
\begin{equation}
\label{eq:q_def}
    \bm{Q}(\bm{u},\bm{r})
    =
    \Bmat\bm{u}
    +
    \bm{J}_1(\bm{u})\bm{u}
    +
    \bm{J}_2(\bm{u})\bm{u}
    +
    \bm{J}_3(\rtau)\bm{u}
    +
    \bm{J}_4(\rtau)
    +
    \qext,
\end{equation}
where
\begin{align}
\label{eq:source_term_components}
\begin{matrix*}[c]
    \bm{J}_1(\bm{u})
    \coloneqq
    \begin{bmatrix}
        -\Lone(\uone)\massmat &\zeromat{6}{6} \\
        \zeromat{6}{6}        &\zeromat{6}{6}
    \end{bmatrix},
    &&&&\bm{J}_2(\bm{u})
    \coloneqq
    \begin{bmatrix}
        \zeromat{6}{6} &-\Ltwo(\utwo)\flexmat \\
        \zeromat{6}{6} & \Lone^\top(\uone)\flexmat
    \end{bmatrix},
    \\[4ex]
    \bm{J}_3(\rtau)
    \coloneqq
    \begin{bmatrix}
        \zeromat{6}{6} &-\Ltwo(\rtau)\flexmat \\
        \zeromat{6}{6} &\zeromat{6}{6}
    \end{bmatrix},
    &&&&\bm{J}_4(\rtau)
    \coloneqq
    \begin{bmatrix}
        \Emat\rtau \\ \zerovec{6}
    \end{bmatrix},
    \\[4ex]
    \Bmat
    \coloneqq
    \begin{bmatrix}
        \zeromat{6}{6} &\Emat \\
        -\Emat^\top    &\zeromat{6}{6}
    \end{bmatrix},
    &&&&\qext
    \coloneqq
    \begin{bmatrix}
        \bm{f}_{e} \\ \zerovec{6}
    \end{bmatrix}
    \phantom{.}
\end{matrix*}
\end{align}
With these definitions,~\eqref{eq:first_order_system_1}--\eqref{eq:first_order_system_2} can then be expressed compactly as
\begin{equation}
\label{eq:capacity_form}
    \Gammamat\dot{\bm{u}} + \Pimat\bm{u}' + \Pimat\bm{r}' = \bm{Q}(\bm{u},\bm{r}),
\end{equation}
and~\eqref{eq:first_order_system_3} can be written as
\begin{equation}
\label{eq:capacity_form_aux}
    \Hmat^{-1}\rtau
    =
    \uone'
    -\Emat^\top\uone
    +\Lone^\top(\uone)\,\flexmat\utwo.
\end{equation}

Many classical mixed formulations introduce auxiliary variables as scaled spatial gradients of the primary unknowns~\cite{ArnoldBrezziCockburnMarini2002}.
Here, the auxiliary variable $\rtau$ in~\eqref{eq:def_rtau} is instead a constitutive quantity: it represents the Kelvin--Voigt damping resultant and therefore combines derivative and zero-order contributions.
This preserves the physical interpretation of the damping while yielding a structure that is well suited for the subsequent energy analysis and DG discretization.

The formulation \eqref{eq:capacity_form}--\eqref{eq:capacity_form_aux} offers two key advantages. First, nonlinear contributions that previously appeared within the advective operator of the hyperbolic--parabolic system \eqref{eq:advection_diffusion_eqn} are now confined to lower-order source terms and the auxiliary relation, so that the main evolution equation involves only linear differential operators in space. This significantly simplifies both the continuous stability analysis and the design of consistent numerical fluxes. Second, apart from the damping-related terms involving $\bm{r}$, the resulting system closely resembles the linear hyperbolic formulation of the undamped intrinsic beam equations derived in~\cite{Bleffert2025}. This structural similarity allows us to reuse essential components of the existing analysis, in particular the treatment of boundary conditions and the construction of the hyperbolic numerical flux, as discussed in Section~\ref{sec:discretization}.

\subsection{Boundary Conditions}
\label{subsec:bc}

We next discuss boundary conditions compatible with the energy structure of the reformulated system.

\paragraph{Undamped (hyperbolic) intrinsic beam equations}
In the undamped case ($\dampmat=\zeromat{6}{6}$), the intrinsic beam equations can be written as a first-order hyperbolic balance law for
$\bm{u}=[\uone^\top,\utwo^\top]^\top\in\R^{12}$,
\begin{equation}
\label{eq:pure_hyperbolic}
    \Gammamat \dot{\bm{u}} + \Pimat \bm{u}' = \bm{Q}_{\mathrm{hyp}}(\bm{u}),
\end{equation}
with source term
\begin{equation}
\label{eq:q_hyp_def}
    \bm{Q}_{\mathrm{hyp}}(\bm{u})
    =
    \Bmat\bm{u}
    +
    \bm{J}_1(\bm{u})\bm{u}
    +
    \bm{J}_2(\bm{u})\bm{u}
    +
    \qext.
\end{equation}
Multiplying~\eqref{eq:pure_hyperbolic} by $\Gammamat^{-1}$ yields the standard first-order form
\begin{equation}
\label{eq:first_order_form_hyp}
    \dot{\bm{u}} + \Amat\,\bm{u}' = \Gammamat^{-1}\bm{Q}_{\mathrm{hyp}}(\bm{u}),
    \qquad
    \Amat := \Gammamat^{-1}\Pimat.
\end{equation}
Since $\Gammamat$ is symmetric positive definite and $\Pimat$ is symmetric,
$\Amat$ has real eigenvalues and a complete set of eigenvectors. Let $\bm{T}=\bm{T}(x)\in\R^{12\times 12}$ be an eigenvector matrix of $\Amat$,
\begin{equation}
    \bm{\Lambda} = \bm{T}^{-1}\Amat \bm{T},
\end{equation}
where $\bm{\Lambda}(x)$ is diagonal and contains the characteristic speeds. The characteristic variables are
$\bm{w}=\bm{T}^{-1}\bm{u}$, and the sign of each characteristic speed determines whether the corresponding component is incoming or outgoing at a given boundary.
Consequently, only the incoming characteristic components require boundary data for a well-posed hyperbolic problem.

In~\cite{Bleffert2025} it is shown that one can impose the incoming characteristic data such that the
boundary conditions take the simple Dirichlet form
\begin{equation}
\label{eq:bc_undamped_dirichlet}
    \uone(0,t)=\bm{\alpha}(t),\qquad \utwo(\ell,t)=\bm{\beta}(t),
\end{equation}
for prescribed functions $\bm{\alpha},\bm{\beta}:[0,T]\to\R^6$. In the \emph{undamped} intrinsic beam
model (i.e., for $\dampmat=\zeromat{6}{6}$), the Kelvin--Voigt term in~\eqref{eq:constitutive_law_forces}
vanishes and the force/moment resultant reduces to
\begin{equation}
    \begin{bmatrix}\bm f\\ \bm m\end{bmatrix}
    = \flexmat^{-1}\begin{bmatrix}\bm\gamma\\ \bm\kappa\end{bmatrix}
    = \utwo.
\end{equation}
Thus, \eqref{eq:bc_undamped_dirichlet} corresponds to prescribing sectional velocities at $x=0$ and
prescribing sectional forces and moments at $x=\ell$.

\paragraph{Damped intrinsic beam equations}
For the damped model, the auxiliary variable $\rtau=\dampmat\dot{\utwo}$ (cf.~\eqref{eq:def_rtau})
enters the evolution through~\eqref{eq:capacity_form} and is linked to $\bm u$ by~\eqref{eq:capacity_form_aux}.
Moreover, the Kelvin--Voigt constitutive law implies that the \emph{total} sectional resultants are
\begin{equation}
    \begin{bmatrix}\bm f\\ \bm m\end{bmatrix}
    = \utwo + \rtau
    = \utwo + \dampmat\,\dot{\utwo}.
\end{equation}
Hence, prescribing forces and moments at $x=\ell$ is no longer equivalent to prescribing $\utwo(\ell,t)$ alone.
Instead, for a prescribed traction function $\bm\beta:[0,T]\to\R^6$ we impose
\begin{equation}
\label{eq:bc_total_resultant}
    \utwo(\ell,t)+\rtau(\ell,t)=\bm\beta(t),
    \qquad t\in(0,T].
\end{equation}
Using $\rtau(\ell,t)=\dampmat(\ell)\dot{\utwo}(\ell,t)$,
\eqref{eq:bc_total_resultant} is equivalently the linear vector-valued ODE
\begin{equation}
\label{eq:bc_ode}
    \dot{\utwo}(\ell,t)
    = \dampmat(\ell)^{-1}\!\left(\bm\beta(t)-\utwo(\ell,t)\right),
    \qquad \utwo(\ell,0)=(\utwo)_0(\ell).
\end{equation}
Equation~\eqref{eq:bc_ode} provides a strong-form interpretation of the traction condition~\eqref{eq:bc_total_resultant}.
Given the compatible initial value, denote the solution of~\eqref{eq:bc_ode}
by $\bm u_{2,\mathrm{bc}}^\ell(t)$ and define
\begin{equation}
\label{eq:bc_split_targets}
    (\rtau)_{\mathrm{bc}}^\ell(t)
    :=\bm\beta(t)-\bm u_{2,\mathrm{bc}}^\ell(t)
    =\dampmat(\ell)\dot{\bm u}_{2,\mathrm{bc}}^\ell(t).
\end{equation}
Thus the elastic and viscous boundary targets are determined by the prescribed
total resultant rather than being independent boundary data. The DG
discretization imposes them separately: $\bm u_{2,\mathrm{bc}}^\ell$ enters the
hyperbolic outer state, while $(\rtau)_{\mathrm{bc}}^\ell$ enters the auxiliary
divergence flux. Their sum satisfies~\eqref{eq:bc_total_resultant} identically;
see Section~\ref{subsec:numerical_flux}.

Combining the reformulated system with the boundary conditions above, we arrive at the following initial--boundary value problem:

\begin{problem}[Initial Boundary Value Problem (\ibvp) for the damped intrinsic beam model]
\label{prob:initial_boundary_value}
Under the coefficient assumptions stated above, given an external load
$\bm f_e:[0,\ell]\times[0,T]\to\R^6$, an initial condition
$\bm u_0=[(\uone)_0^\top,(\utwo)_0^\top]^\top:[0,\ell]\to\R^{12}$, and boundary
data $\bm{\alpha},\bm{\beta}:[0,T]\to\R^6$ that are compatible at $t=0$, find
sufficiently regular fields $\bm{u}:[0,\ell]\times[0,T]\to\R^{12}$ and
$\rtau:[0,\ell]\times[0,T]\to\R^{6}$ such that
\begin{align}
\label{eq:initial_boundary_problem}
    \begin{cases}
        \Gammamat \dot{\bm{u}} + \Pimat \bm{u}' + \Pimat \bm{r}' = \bm{Q}(\bm{u},\bm{r})
        &\text{in } (0,\ell)\times(0,T],\\[0.75ex]
        \Hmat^{-1}\rtau
        = \uone' - \Emat^\top\uone + \Lone^\top(\uone)\,\flexmat\utwo
        &\text{in } (0,\ell)\times(0,T],\\[0.75ex]
        \uone(0,t)=\bm{\alpha}(t)
        &\text{for } t\in(0,T],\\[0.75ex]
        \utwo(\ell,t)+\rtau(\ell,t)=\bm{\beta}(t)
        &\text{for } t\in(0,T],\\[0.75ex]
        \bm{u}(x,0)=\bm{u}_0(x)
        &\text{for } x\in[0,\ell],
    \end{cases}
\end{align}
\end{problem}

\section{Continuous Analysis}
\label{sec:continuous_analysis}

To design an energy-stable DG discretization for the
\ibvp~\eqref{eq:initial_boundary_problem}, we first derive a continuous
mechanical energy identity for the reformulated system. For sufficiently
regular solutions, this identity yields an \textit{a priori} bound for the mechanical
energy $E(t)$ defined in Eq.~\eqref{eq:physical_energy}.

Energy estimates for the undamped intrinsic beam equations were derived by Hodges~\cite{Hodges2003}, and Artola et al.~\cite{Artola2021} extended the analysis to the Kelvin--Voigt damped model. Here we adapt these arguments to the reformulation in Problem~\ref{prob:initial_boundary_value} and derive an energy identity that will later be mimicked at the discrete level.

Let $\Omega=[0,\ell]$ and let $\inproduct{\cdot}{\cdot}$ denote the
$L^2(\Omega)$ inner product. For a fixed final time $T>0$, define
\begin{align}
\label{eq:continuous_solution_spaces}
    \mathcal X_T
    &:={}
    H^1\!\left(0,T;L^2(\Omega)^{12}\right)
    \cap
    L^\infty\!\left(0,T;H^1(\Omega)^{12}\right),
    &
    \mathcal Y_T
    &:={}
    L^2\!\left(0,T;H^1(\Omega)^6\right).
\end{align}
We use the mechanical-energy norm and the dual weighted norm
\begin{align}
\label{eq:continuous_energy_norm}
    \engnorm{\bm z}^{2}
    &:={}
    \inproduct{\bm z}{\Gammamat\bm z},
    &
    \engnorm{\bm u(\cdot,t)}^{2}
    &=2E(t),
    \\
\label{eq:continuous_forcing_norm}
    \|\bm g\|_{\Gammamat^{-1}}^{2}
    &:={}
    \int_{\Omega}\bm g^{\top}\Gammamat^{-1}\bm g\,\mathrm{d}x.
\end{align}
Uniform positive definiteness of $\Gammamat$ makes the norms on weighted spaces
$L^2_{\Gammamat}(\Omega)^{12}$ and
$L^2_{\Gammamat^{-1}}(\Omega)^{12}$ equivalent to the standard
$L^2(\Omega)^{12}$ norm.

At each endpoint $x_\delta\in\{0,\ell\}$, fix a partition
$\mathcal I_v(x_\delta)\mathbin{\dot\cup}\mathcal I_t(x_\delta)
=\{1,\ldots,6\}$. The corresponding homogeneous zero-boundary-power
conditions are
\begin{equation}
\label{eq:natural_bc_dib}
    (\uone)_i(x_\delta,t)=0
    \quad(i\in\mathcal I_v(x_\delta)),
    \qquad
    \bigl(\utwo+\rtau\bigr)_i(x_\delta,t)=0
    \quad(i\in\mathcal I_t(x_\delta)).
\end{equation}
For a cantilever,
$\mathcal I_v(0)=\mathcal I_t(\ell)=\{1,\ldots,6\}$, with the complementary
sets empty. Condition~\eqref{eq:natural_bc_dib} is a fixed boundary
prescription, not the under-specified nonlinear product condition
$(\uone)_i(\utwo+\rtau)_i=0$. The energy calculation below applies to any such
fixed partition; well-posedness for arbitrary partitions would require a
separate initial--boundary value analysis. Unless stated otherwise, we use the
homogeneous cantilever partition.

\pagebreak[3]
\begin{theorem}[Continuous mechanical energy balance and stability]
\label{thm:continuous_energy}
Assume the coefficient and structural hypotheses of
Section~\ref{sec:damped_intrinsic_beam_equations}; in particular, assume that
$\Gammamat$ and $\Hmat$ are time-independent and uniformly symmetric positive
definite on $\Omega$. Let
$\qext\in L^2(0,T;L^2_{\Gammamat^{-1}}(\Omega)^{12})$, and let
$(\bm u,\rtau)\in\mathcal X_T\times\mathcal Y_T$ be a sufficiently regular
solution whose two field equations in
Problem~\ref{prob:initial_boundary_value} hold almost everywhere, with
$\bm u(\cdot,0)=\bm u_0$. Then $E$ is absolutely continuous and, for almost
every $t\in(0,T)$,
\begin{equation}
\label{eq:energy_result_general}
    \frac{1}{2}\frac{\mathrm{d}}{\mathrm{d}t}\engnorm{\bm{u}}^2
    +
    \inproduct{\Hmat^{-1}\rtau}{\rtau}
    =
    \bndeval{\uone}{(\utwo+\rtau)}
    +
    \inproduct{\qext}{\bm{u}}.
\end{equation}
If, in addition, the endpoint traces satisfy the fixed homogeneous conditions
\eqref{eq:natural_bc_dib}, then
\begin{equation}
\label{eq:final_bound}
    \engnorm{\bm{u}(\cdot,T)}^{2}
    +
    2\int_{0}^{T}
    \inproduct{\Hmat^{-1}\rtau(\cdot,t)}{\rtau(\cdot,t)}\,\mathrm{d}t
    \le
    e^{T}\left(
        \engnorm{\bm{u}_0}^{2}
        +
        \int_{0}^{T}\|\qext(\cdot,t)\|_{\Gammamat^{-1}}^{2}\,\mathrm{d}t
    \right).
\end{equation}
\end{theorem}

\begin{proof}
For almost every $t\in(0,T)$, let $\bm{v}\in H^1(\Omega)^{12}$ and
$\bm{w}\in H^1(\Omega)^{6}$ be test functions. Multiplying
\eqref{eq:capacity_form} by $\bm{v}$ and \eqref{eq:capacity_form_aux} by
$\bm{w}$, integrating over $\Omega$, and applying integration by parts yields
the global weak formulation
\begin{align}
\label{eq:global_weak_formulation}
    \inproduct{\Gammamat\dot{\bm{u}}}{\bm{v}}
    &=
    \linformA{\bm{u}}{\bm{v}}
    +\linformB{\bm{r}}{\bm{v}}
    +\inproduct{\bm{Q}(\bm{u},\bm{r})}{\bm{v}},
    \\[1.5ex]
\label{eq:global_weak_formulation_aux}
    \inproduct{\Hmat^{-1}\rtau}{\bm{w}}
    &=
    \linformC{\uone}{\bm{w}}
    -\inproduct{\Emat^\top\uone}{\bm{w}}
    +\inproduct{\Lone^\top(\uone)\flexmat\utwo}{\bm{w}},
\end{align}
where
\begin{align}
\label{eq:linforms_def}
    \linformA{\bm{u}}{\bm{v}}
    &:=
    \inproduct{\Pimat \bm{u}}{\bm{v}'}
    -\bndevallong{\Pimat \bm{u}}{\bm{v}},
    &
    \linformB{\bm{r}}{\bm{v}}
    &:=
    \inproduct{\Pimat \bm{r}}{\bm{v}'}
    -\bndevallong{\Pimat \bm{r}}{\bm{v}},
    \\
    \linformC{\uone}{\bm{w}}
    &:=
    \bndeval{\uone}{\bm{w}}
    -\inproduct{\uone}{\bm{w}'}.
\end{align}

Choosing $\bm{v}=\bm{u}$ and $\bm{w}=\rtau$ in \eqref{eq:global_weak_formulation}--\eqref{eq:global_weak_formulation_aux} and adding the two identities yields
\begin{align}
\label{eq:energy_all_terms}
\begin{split}
    \frac{1}{2}\frac{\mathrm{d}}{\mathrm{d}t}\engnorm{\bm{u}}^2
    +\inproduct{\Hmat^{-1}\rtau}{\rtau}
    =
    &\linformA{\bm{u}}{\bm{u}}
    +\linformB{\bm{r}}{\bm{u}}
    +\inproduct{\bm{Q}(\bm{u},\bm{r})}{\bm{u}}
    \\
    &+\linformC{\uone}{\rtau}
    -\inproduct{\Emat^\top\uone}{\rtau}
    +\inproduct{\Lone^\top(\uone)\flexmat\utwo}{\rtau}.
\end{split}
\end{align}
The norm in~\eqref{eq:continuous_energy_norm} is induced by the kinetic and
elastic energy, so Eq.~\eqref{eq:energy_all_terms} describes the evolution of
the mechanical energy.

\paragraph{First-order spatial operator}
Using symmetry of $\Pimat$ and integration by parts, we obtain
\begin{equation}
    \inproduct{\Pimat\bm{u}}{\bm{u}'}
    = \frac12\,\bndevallong{\Pimat\bm{u}}{\bm{u}},
\end{equation}
and therefore
\begin{equation}
    \linformA{\bm{u}}{\bm{u}}
    = -\frac12\,\bndevallong{\Pimat\bm{u}}{\bm{u}}
    = \bndeval{\uone}{\utwo},
\end{equation}
where the last equality follows from the block structure of $\Pimat$ and $\bm{u}=[\uone^\top,\utwo^\top]^\top$.
Thus, $\linformA{\bm{u}}{\bm{u}}$ contributes only through boundary work.

\paragraph{Auxiliary coupling terms}
Using the definitions of $\Pimat$ and $\bm{r}$, we have
\begin{equation}
    \linformB{\bm{r}}{\bm{u}}
    =
    \bndeval{\rtau}{\uone}
    -
    \inproduct{\rtau}{\uone'}.
\end{equation}
Adding $\linformC{\uone}{\rtau}=\bndeval{\uone}{\rtau}-\inproduct{\uone}{\rtau^\prime}$ gives
\begin{equation}
    \linformB{\bm{r}}{\bm{u}}+\linformC{\uone}{\rtau}
    =
    2\,\bndeval{\uone}{\rtau}
    -
    \inproduct{\rtau}{\uone^\prime}
    -
    \inproduct{\uone}{\rtau^\prime}.
\end{equation}
Integrating by parts in the last term,
$\inproduct{\uone}{\rtau^\prime}=\bndeval{\uone}{\rtau}-\inproduct{\uone^\prime}{\rtau}$,
yields the cancellation
\begin{equation}
    \linformB{\bm{r}}{\bm{u}}+\linformC{\uone}{\rtau}
    =
    \bndeval{\uone}{\rtau},
\end{equation}
so the combined contribution of $\linformB{\bm{r}}{\bm{u}}$ and $\linformC{\uone}{\rtau}$  is again purely a boundary term.

\paragraph{Lower-order terms}
We next simplify the remaining source contribution,
$\inproduct{\bm{Q}(\bm{u},\bm{r})}{\bm{u}}$.
Using~\eqref{eq:q_def} and~\eqref{eq:source_term_components}, we expand
\begin{align}
    \inproduct{\bm{Q}(\bm{u},\bm{r})}{\bm{u}}
    &=
    \inproduct{\Bmat\bm{u}}{\bm{u}}
    +
    \inproduct{\bm{J}_1(\bm{u})\bm{u}}{\bm{u}}
    +
    \inproduct{\bm{J}_2(\bm{u})\bm{u}}{\bm{u}}
    +
    \inproduct{\bm{J}_3(\rtau)\bm{u}}{\bm{u}}
    +
    \inproduct{\bm{J}_4(\rtau)}{\bm{u}}
    +
    \inproduct{\qext}{\bm{u}}.
\end{align}

First, since  $\Bmat$ is skew-symmetric $\Bmat^\top=-\Bmat$, we have $\inproduct{\Bmat\bm{u}}{\bm{u}}=0$.
Next, by the definition of $\bm{J}_1$,
\begin{equation}
    \inproduct{\bm{J}_1(\bm{u})\bm{u}}{\bm{u}}
    =
    -\inproduct{\Lone(\uone)\massmat\uone}{\uone}
    =
    -\inproduct{\massmat\uone}{\Lone^\top(\uone)\uone}
    =0,
\end{equation}
using the identity~\eqref{eq:l_operator_properties_1} $\Lone^\top(\bm{c})\bm{c}=\zerovec{6}$.
For the $\bm{J}_2$ term, we obtain
\begin{align}
    \inproduct{\bm{J}_2(\bm{u})\bm{u}}{\bm{u}}
    &=
    -\inproduct{\Ltwo(\utwo)\flexmat\utwo}{\uone}
    +
    \inproduct{\Lone^\top(\uone)\flexmat\utwo}{\utwo}
    \\
    &=
    -\inproduct{\Ltwo(\utwo)\flexmat\utwo}{\uone}
    +
    \inproduct{\flexmat\utwo}{\Lone(\uone)\utwo}
    \\
    &=
    -\inproduct{\Ltwo(\utwo)\flexmat\utwo}{\uone}
    +
    \inproduct{\flexmat\utwo}{\Ltwo^\top(\utwo)\uone}
    =0,
\end{align}
where we used the identity~\eqref{eq:l_operator_properties_2}
, i.e. $\Lone (\uone)\utwo = \Ltwo^\top(\utwo)\uone$ in the last step.

For the damping-dependent contribution $\bm{J}_3$, the block structure gives
\begin{equation}
    \inproduct{\bm{J}_3(\rtau)\bm{u}}{\bm{u}}
    =
    -\inproduct{\Ltwo(\rtau)\flexmat\utwo}{\uone}
    =
    -\inproduct{\Lone^\top(\uone)\flexmat\utwo}{\rtau},
\end{equation}
again by~\eqref{eq:l_operator_properties_2} and symmetry of the $L^2$ inner product.
Finally,
\begin{equation}
    \inproduct{\bm{J}_4(\rtau)}{\bm{u}}
    =
    \inproduct{\Emat\rtau}{\uone}
    =
    \inproduct{\Emat^\top\uone}{\rtau}.
\end{equation}

Collecting these identities, we arrive at
\begin{equation}
\label{eq:source_term_simplified}
    \inproduct{\bm{Q}(\bm{u},\bm{r})}{\bm{u}}
    =
    -\inproduct{\Lone^\top(\uone)\flexmat\utwo}{\rtau}
    +
    \inproduct{\Emat^\top\uone}{\rtau}
    +
    \inproduct{\qext}{\bm{u}}.
\end{equation}

Each term in \eqref{eq:source_term_simplified}, except for the forcing term
$\inproduct{\qext}{\bm{u}}$, cancels with a corresponding term in
\eqref{eq:energy_all_terms}. Substitution into~\eqref{eq:energy_all_terms}
gives the boundary-work identity~\eqref{eq:energy_result_general}, which proves
the first assertion.

\paragraph{Boundary conditions}
Under~\eqref{eq:natural_bc_dib}, every component of
$\uone^\top(\utwo+\rtau)$ vanishes at both endpoints, hence
\(
    \bndeval{\uone}{(\utwo+\rtau)}=0.
\)
The boundary term in~\eqref{eq:energy_result_general} then vanishes and we obtain
\begin{equation}
\label{eq:combined_energy_bc_0}
    \frac{1}{2}\frac{\mathrm{d}}{\mathrm{d}t} \engnorm{\bm{u}}^2
    +
    \inproduct{\Hmat^{-1}\rtau}{\rtau}
    =
    \inproduct{\bm{u}}{\qext}.
\end{equation}

\paragraph{Forcing estimate and Grönwall bound}
Since $\Gammamat$ is symmetric positive definite for all $x\in\Omega$, the (principal) matrix functions
$\Gammamat^{1/2}(x)$ and $\Gammamat^{-1/2}(x)$ are well-defined and symmetric positive definite.
Using $\Gammamat^{1/2}\Gammamat^{-1/2}=I$ pointwise, we rewrite the forcing term as
\begin{equation}
    \inproduct{\bm{u}}{\qext}
    =
    \inproduct{\Gammamat^{1/2}\bm{u}}{\Gammamat^{-1/2}\qext}.
\end{equation}
By Cauchy--Schwarz and Young's inequality,
\begin{equation}
    \inproduct{\Gammamat^{1/2}\bm{u}}{\Gammamat^{-1/2}\qext}
    \le
    \frac12 \|\Gammamat^{1/2}\bm{u}\|_{L^2(\Omega)}^{2}
    +
    \frac12 \|\Gammamat^{-1/2}\qext\|_{L^2(\Omega)}^{2}
    =
    \frac12 \engnorm{\bm{u}}^{2}
    +
    \frac12 \|\qext\|_{\Gammamat^{-1}}^{2}.
\end{equation}
Substituting this bound into~\eqref{eq:combined_energy_bc_0} and multiplying by $2$ yields
\begin{equation}
\label{eq:stability_bound_1}
    \frac{\mathrm{d}}{\mathrm{d}t} \engnorm{\bm{u}}^{2}
    +
    2\,\inproduct{\Hmat^{-1}\rtau}{\rtau}
    \le
    \engnorm{\bm{u}}^{2}
    +
    \|\qext\|_{\Gammamat^{-1}}^{2}.
\end{equation}

Integrating~\eqref{eq:stability_bound_1} over $[0,T]$ gives
\begin{equation}
\label{eq:stability_bound_2}
    \engnorm{\bm{u}(\cdot,T)}^{2}
    +
    2\int_{0}^{T}\inproduct{\Hmat^{-1}\rtau(\cdot,t)}{\rtau(\cdot,t)}\,\mathrm{d}t
    \le
    \engnorm{\bm{u}_0}^{2}
    +
    \int_{0}^{T}\engnorm{\bm{u}(\cdot,t)}^{2}\,\mathrm{d}t
    +
    \int_{0}^{T}\|\qext(\cdot,t)\|_{\Gammamat^{-1}}^{2}\,\mathrm{d}t.
\end{equation}

Define
\begin{equation}
    z(t)
    :=
    \engnorm{\bm{u}(\cdot,t)}^{2}
    +
    2\int_{0}^{t}\inproduct{\Hmat^{-1}\rtau(\cdot,s)}{\rtau(\cdot,s)}\,\mathrm{d}s.
\end{equation}
Then $z(0)=\engnorm{\bm{u}_0}^{2}$ and $\engnorm{\bm{u}(\cdot,t)}^{2}\le z(t)$, so~\eqref{eq:stability_bound_2} implies
\begin{equation}
    z(T)
    \le
    z(0)
    +
    \int_{0}^{T} z(t)\,\mathrm{d}t
    +
    \int_{0}^{T}\|\qext(\cdot,t)\|_{\Gammamat^{-1}}^{2}\,\mathrm{d}t.
\end{equation}
Grönwall's inequality now gives~\eqref{eq:final_bound}, completing the proof.
\end{proof}

In particular, if \(\qext=\bm0\) and the boundary data are homogeneous, then
\eqref{eq:combined_energy_bc_0} shows that the mechanical energy is
non-increasing.
This estimate shows that, under homogeneous natural boundary conditions and in
the presence of bounded external forces and moments, the mechanical energy of
any sufficiently regular solution is bounded. In particular, the damping
contribution provides a non-negative dissipation term, while the growth of the
energy is controlled by the initial data and the total energy input of the
external loads.

Having established the continuous energy estimate, we now proceed to the spatial discretization. In the next section, we derive a DG formulation that mimics the structure of \eqref{eq:combined_energy_bc_0} at the semi-discrete level and prove a corresponding discrete energy inequality.

\section{Discretization}
\label{sec:discretization}

In this section, we derive a semi-discrete DG formulation of the weak problem
\eqref{eq:global_weak_formulation}--\eqref{eq:global_weak_formulation_aux} for the damped intrinsic beam system.

We partition $\Omega$ into a mesh $\mathcal{T}_h$ of $N_c$ non-overlapping elements
\begin{equation}
  \mathcal{T}_h = \bigcup_{i=1}^{N_c} \Omega_i,
  \qquad
  \Omega_i = [x_{i-1}, x_{i}],
\end{equation}
with interior interfaces $\{x_i\}_{i\in\mathcal{I}_\mathrm{int}}$ indexed by $\mathcal{I}_{\mathrm{int}}:=\{1,\dots,N_c-1\}$, and physical boundaries
at $x_0=0$ and $x_{N_c}=\ell$, with $\mathcal{I}_{\mathrm{ext}}:=\{0,N_c\}$.

On each element $\Omega_i$, we approximate $\bm u$ and $\rtau$ by polynomials of degree at most $k$,
\begin{equation}
  \uh|_{\Omega_i} \in P_k^{12}(\Omega_i),
  \qquad
  \rtauh|_{\Omega_i} \in P_k^{6}(\Omega_i),
\end{equation}
and define the discrete auxiliary vector
\begin{equation}
    \bm{r}_h :=
    \begin{bmatrix}
        \zerovec{6} \\
        \rtauh
    \end{bmatrix}.
\end{equation}
Likewise, we take test functions $\bm v_h|_{\Omega_i}\in P_k^{12}(\Omega_i)$ and $\bm w_h|_{\Omega_i}\in P_k^{6}(\Omega_i)$.

\paragraph{Elementwise semi-discrete formulation}
On each element $\Omega_i$, the semi-discrete DG formulation reads
\begin{align}
\label{eq:local_semi_discrete_formulation}
    &\inproduct[\Omega_i]{\Gammamat \dot{\uh}}{\bm{v}_h}
    + \left[ \widehat{\Pimat \uh}^\top \bm{v}_h \right]_{x_{i-1}}^{x_{i}}
    - \inproduct[\Omega_i]{\Pimat \uh}{\bm{v}_h^\prime}
    + \left[ \widehat{\Pimat \bm{r}_h}^\top \bm{v}_h \right]_{x_{i-1}}^{x_{i}}
    - \inproduct[\Omega_i]{\Pimat \bm{r}_h}{\bm{v}_h^\prime}
    = \inproduct[\Omega_i]{\bm{Q}(\uh,\bm{r}_h)}{\bm{v}_h}, \notag
    \\[1.25ex]
    &\inproduct[\Omega_i]{\Hmat^{-1}\rtauh}{\bm{w}_h}
    = \left[ \widehat{\uoneh}^\top \bm{w}_h \right]_{x_{i-1}}^{x_{i}}
    - \inproduct[\Omega_i]{\uoneh}{\bm{w}_h^\prime}
    - \inproduct[\Omega_i]{\Emat^\top \uoneh}{\bm{w}_h}
    + \inproduct[\Omega_i]{\Lone^\top(\uoneh) \flexmat \utwoh}{\bm{w}_h}.
\end{align}
Here $\widehat{\Pimat \uh}$, $\widehat{\Pimat \bm{r}_h}$, and $\widehat{\uoneh}$
denote numerical fluxes evaluated at $x_{i-1}$ and $x_i$, depending on the interior/exterior traces (and on prescribed outer states at $x=0,\ell$).
Precise definitions are given in Section~\ref{subsec:numerical_flux}.

\subsection{Numerical Fluxes and Boundary Treatment}
\label{subsec:numerical_flux}

To close the semi-discrete formulation, we specify numerical fluxes for the first-order spatial
operator and for the auxiliary coupling terms. These fluxes also provide a weak imposition of
the boundary conditions in IBVP~\ref{prob:initial_boundary_value}.

\paragraph{Interface traces, jumps, and averages}
Let $w$ be a scalar-, vector-, or matrix-valued (piecewise smooth) function. At an interior
interface $x_i$, $i\in\mathcal{I}_{\mathrm{int}}$, we define the left and right traces by
\begin{equation}
    w^{-}(x_i) := \lim_{\epsilon \to 0^+} w(x_i-\epsilon),
    \qquad
    w^{+}(x_i) := \lim_{\epsilon \to 0^+} w(x_i+\epsilon).
\end{equation}
The jump and average operators are
\begin{equation}
    \jump{w} := w^- - w^+,
    \qquad
    \avg{w} := \tfrac{1}{2}\left(w^- + w^+\right).
\end{equation}
We will also use the standard product identity
\begin{equation}
\label{eq:jump_avg_property}
    \jump{ab} = \avg{a}\,\jump{b} + \jump{a}\,\avg{b},
\end{equation}
whenever the products are well-defined.

At the physical boundaries $x\in\{0,\ell\}$, the interior trace is taken from the adjacent element, and the exterior trace is defined via the outer states.

\paragraph{Flux for the first-order spatial operator}
The first-order spatial operator in \eqref{eq:local_semi_discrete_formulation} is governed by
$\Pimat\bm u$. Since the propagative (hyperbolic) part of the Kelvin--Voigt damped intrinsic
beam system has the same characteristic structure as the undamped intrinsic beam equations,
we adopt the Lax--Friedrichs-type numerical flux used in~\cite{Bleffert2025}; since the penalty employs the full characteristic matrix $\Gammamat|\Amat|$ rather than a scalar wave-speed bound, $\sigma=1$ yields the exact characteristic upwind flux. Specifically, at any interface point $x_i$ we define
\begin{equation}
\label{eq:flux_pi_u}
    \widehat{\Pimat \uh}(x_i)
    :=
    \Pimat \avg{\uh}(x_i)
    +
    \frac{\sigma}{2}\,\Gammamat |\Amat|\,\jump{\uh}(x_i),
\end{equation}
where $\sigma\ge 0$ is a stabilization parameter.
The choice $\sigma=0$ yields the central flux, whereas $\sigma=1$ gives the fully upwind (characteristic) flux.

The matrix $|\Amat|$ is defined through an eigenvalue decomposition
$\Amat=\bm T\bm\Lambda\bm T^{-1}$ by
\begin{equation}
    |\Amat|:=\bm T|\bm\Lambda|\bm T^{-1},
\end{equation}
where $|\bm\Lambda|$ is obtained by taking absolute values of the diagonal entries of $\bm\Lambda$.
For the smooth material coefficients assumed here, their left and right traces
coincide, so \(\Gammamat|\Amat|\) is evaluated at that unique interface value.
Discontinuous material interfaces would require a separate two-material flux
and are not covered by the present estimate.

At interior interfaces, $\sigma\in[0,1]$ trades accuracy against numerical dissipation; at the
physical boundaries $x=0$ and $x=\ell$ we always set $\sigma=1$, which, combined with the outer
states prescribed below, yields the dissipative boundary terms of
Section~\ref{subsec:discrete_energy_stability}.

\paragraph{Fluxes for the auxiliary coupling}
The first-order flux \(\widehat{\Pimat\uh}\) and the two fluxes in the
Kelvin--Voigt auxiliary pair are independent numerical choices.  For the latter
we use the alternating LDG orientation~\cite{CockburnShu1998a}
\begin{equation}
\label{eq:flux_aux}
    \widehat{\uoneh}:=\uoneh^-,
    \qquad
    \widehat{\rtauh}:=\rtauh^+,
    \qquad
    \widehat{\Pimat\bm r_h}
    :=\Pimat
    \begin{bmatrix}\zerovec{6}\\ \widehat{\rtauh}\end{bmatrix}
    =
    \begin{bmatrix}-\widehat{\rtauh}\\ \zerovec{6}\end{bmatrix}.
\end{equation}
Thus ``upwind'' below refers to the characteristic first-order flux
\eqref{eq:flux_pi_u}, whereas ``alternating'' refers to the complementary
one-sided traces in~\eqref{eq:flux_aux}.  Reversing both sides globally gives
the other standard alternating orientation; we use~\eqref{eq:flux_aux} because
it aligns the left essential datum for \(\uone\) with the right natural datum
for \(\rtau\).

For the controlled BR1 comparison~\cite{bassi1997high} in
Section~\ref{subsec:mms}, the interior traces are instead the central pair
\(\widehat{\uoneh}=\avg{\uoneh}\) and
\(\widehat{\rtauh}=\avg{\rtauh}\). Both choices satisfy
\begin{equation}
\label{eq:aux_trace_compatibility}
    \widehat{\rtauh}^{\top}\jump{\uoneh}
    +\widehat{\uoneh}^{\top}\jump{\rtauh}
    =\jump{\uoneh^{\top}\rtauh}.
\end{equation}
by direct expansion for~\eqref{eq:flux_aux} and by the jump--average product
identity~\eqref{eq:jump_avg_property} for BR1. Thus both are neutral in the
auxiliary-interface energy contraction; this does not assert identical
accuracy or convergence of the reconstructed \(\rtauh\).

\paragraph{Boundary treatment via outer states}
\label{par:outer_states}

At the physical boundaries, numerical fluxes are evaluated using an \emph{interior trace}
(from the adjacent element) and a \emph{fictitious exterior state} (ghost state).
We denote the interior trace by $\uh^{+}(0,t)$ at $x=0$ and by $\uh^{-}(\ell,t)$ at $x=\ell$.
The corresponding outer (ghost) states are $\uh^{-}(0,t)$ and $\uh^{+}(\ell,t)$, respectively.

We consider boundary data consistent with Problem~\ref{prob:initial_boundary_value},
\begin{equation}
\bm u_1(0,t)=\bm\alpha(t),
\qquad
(\bm u_2+\rtau)(\ell,t)=\bm\beta(t).
\end{equation}
The second condition prescribes the \emph{total} (Kelvin--Voigt) sectional resultants at $x=\ell$.
We use the continuum-derived split~\eqref{eq:bc_split_targets}: the hyperbolic
outer state below receives $\bm u_{2,\mathrm{bc}}^\ell$, and the auxiliary
divergence flux receives $(\rtau)_{\mathrm{bc}}^\ell$.

Define
\begin{equation}
\Psimat := \flexmat^{-1/2}\massmat^{-1}\flexmat^{-1/2},
\qquad
\bm S_0 := \flexmat^{-1/2}\Psimat^{-1/2}\flexmat^{-1/2},
\qquad
\bm S_\ell := \flexmat^{1/2}\Psimat^{1/2}\flexmat^{1/2}.
\end{equation}
Since $\flexmat$ and $\massmat$ are symmetric positive definite, so is $\Psimat$, and the
(principal) matrix square roots $\Psimat^{\pm1/2}$ are well-defined; hence
$\bm S_0$ and $\bm S_\ell$ are symmetric positive definite as well. Moreover,
$\bm S_0\bm S_\ell=\bm S_\ell\bm S_0=\idmat{6}{6}$, an identity used repeatedly in the boundary
analysis below.

Following the characteristic outer-state construction used for the intrinsic beam equations
(cf.~\cite{Bleffert2025}), we prescribe the ghost states as
\begin{align}
\label{eq:outer_states}
\uh^{-}(0,t)
=
\begin{bmatrix}
\bm\alpha(t) \\
\utwoh^{+}(0,t) + \bm S_0\bigl(\uoneh^{+}(0,t)-\bm\alpha(t)\bigr)
\end{bmatrix},
\qquad
\uh^{+}(\ell,t)
=
\begin{bmatrix}
\uoneh^{-}(\ell,t) - \bm S_\ell\bigl(\utwoh^{-}(\ell,t)-\bm u_{2,\mathrm{bc}}^\ell(t)\bigr) \\
\bm u_{2,\mathrm{bc}}^\ell(t)
\end{bmatrix}.
\end{align}

By construction, these choices enforce the prescribed boundary data in the limit of mesh refinement.
For homogeneous boundary data ($\bm\alpha\equiv\zerovec{6}$, $\bm\beta\equiv\zerovec{6}$) they moreover
yield dissipative boundary contributions in the discrete energy balance (see
Section~\ref{subsec:discrete_energy_stability}); the inhomogeneous case is discussed in
Remark~\ref{rem:inhomogeneous_bc_data}.

\subsection{Global Semi-Discrete Formulation}
\label{subsec:global_semidiscrete}

The global semi-discrete DG formulation is obtained by assembling the elementwise
relations~\eqref{eq:local_semi_discrete_formulation} over all elements
$\{\Omega_i\}_{i=1}^{N_c}$.

\begin{problem}[Global semi-discrete DG formulation]
\label{prob:global_formulation}
Let \(T>0\) and define
\begin{align}
V_h^k &:= \bigl\{\bm v_h\in L^2(\Omega)^{12}\,\big|\, \bm v_h|_{\Omega_i}\in P_k^{12}(\Omega_i),
\ i=1,\dots,N_c\bigr\}, 
\\
W_h^k &:= \bigl\{\bm w_h\in L^2(\Omega)^{6}\,\big|\, \bm w_h|_{\Omega_i}\in P_k^{6}(\Omega_i),
\ i=1,\dots,N_c\bigr\}.
\end{align}
Given \(\bm u_{h,0}\in V_h^k\), we seek
\begin{equation}
\label{eq:semidiscrete_solution_spaces}
    \uh\in H^1(0,T;V_h^k),
    \qquad
    \rtauh\in L^2(0,T;W_h^k),
\end{equation}
with
\begin{equation}
\label{eq:semidiscrete_initial_condition}
    \uh(\cdot,0)=\bm u_{h,0},
\end{equation}
such that, for almost every \(t\in(0,T)\) and all
$(\bm v_h,\bm w_h)\in V_h^k\times W_h^k$,
\begin{align}
\label{eq:global_1_hat}
\sum_{i=1}^{N_c} \inproduct[\Omega_i]{\Gammamat \dot{\uh}}{\bm{v}_h}
&=
\dlinformA{\uh}{\bm{v}_h}
+\dlinformB{\bm{r}_h}{\bm{v}_h}
+\sum_{i=1}^{N_c}\inproduct[\Omega_i]{\bm{Q}(\uh,\bm{r}_h)}{\bm{v}_h},
\\[0.5ex]
\label{eq:global_2_hat}
\sum_{i=1}^{N_c} \inproduct[\Omega_i]{\Hmat^{-1}\rtauh}{\bm{w}_h}
&=
\dlinformC{\uoneh}{\bm{w}_h}
-\sum_{i=1}^{N_c}\inproduct[\Omega_i]{\Emat^\top \uoneh}{\bm{w}_h}
+\sum_{i=1}^{N_c}\inproduct[\Omega_i]{\Lone^\top(\uoneh)\flexmat\utwoh}{\bm{w}_h},
\end{align}
where $\bm r_h := [\zerovec{6}^\top,\rtauh^\top]^\top$ and the semi-discrete forms are
\begin{align}
\dlinformA{\uh}{\bm{v}_h}
:=\;&
\sum_{i=1}^{N_c}\inproduct[\Omega_i]{\Pimat \uh}{\bm{v}_h^\prime}
-\sum_{i\in\mathcal I_{\mathrm{int}}}\interfacepoint[x_i]{\widehat{\Pimat \uh}}{\jump{\bm v_h}}
-\externalboundary{\widehat{\Pimat \uh}}{\bm v_h},
\\
\dlinformB{\bm{r}_h}{\bm{v}_h}
:=\;&
\sum_{i=1}^{N_c}\inproduct[\Omega_i]{\Pimat \bm{r}_h}{\bm{v}_h^\prime}
-\sum_{i\in\mathcal I_{\mathrm{int}}}\interfacepoint[x_i]{\widehat{\Pimat \bm r_h}}{\jump{\bm v_h}}
-\externalboundary{\widehat{\Pimat \bm r_h}}{\bm v_h},
\\
\dlinformC{\uoneh}{\bm{w}_h}
:=\;&
\sum_{i\in\mathcal I_{\mathrm{int}}}\interfacepointsingle[x_i]{\widehat{\uoneh}}{\jump{\bm w_h}}
+\externalboundarysingle{\widehat{\uoneh}}{\bm w_h}
-\sum_{i=1}^{N_c}\inproduct[\Omega_i]{\uoneh}{\bm{w}_h^\prime}.
\end{align}
Here $\widehat{\Pimat \uh}$, $\widehat{\Pimat \bm r_h}$ and $\widehat{\uoneh}$
are numerical fluxes defined in Section~\ref{subsec:numerical_flux}, using interior traces
and the outer states at $x=0,\ell$.
\end{problem}

\subsection{Energy Stability}
\label{subsec:discrete_energy_stability}

We prove that the semi-discrete DG scheme satisfies an energy--dissipation identity
analogous to the continuous relation~\eqref{eq:combined_energy_bc_0}.

Mirroring the continuous analysis (cf.\ the discussion preceding~\eqref{eq:combined_energy_bc_0}),
we assume \emph{homogeneous boundary data} throughout this subsection:
$\bm\alpha\equiv\zerovec{6}$ and $\bm\beta\equiv\zerovec{6}$, realized by the trivial splitting
$\bm u_{2,\mathrm{bc}}^\ell\equiv\zerovec{6}$ and $(\rtau)_{\mathrm{bc}}^\ell\equiv\zerovec{6}$.
The outer states~\eqref{eq:outer_states} then reduce to
\begin{equation}
\label{eq:homogeneous_outer_states}
\uh^{-}(0,t)
=
\begin{bmatrix}
\zerovec{6}\\[2pt]
\utwoh^{+} + \bm S_0\,\uoneh^{+}
\end{bmatrix},
\qquad
\uh^{+}(\ell,t)
=
\begin{bmatrix}
\uoneh^{-} - \bm S_\ell\,\utwoh^{-}\\[2pt]
\zerovec{6}
\end{bmatrix},
\end{equation}
which are exactly the states substituted in
\ref{app:boundary_contribution_analysis}. The inhomogeneous case is addressed in
Remark~\ref{rem:inhomogeneous_bc_data} below.

For both auxiliary-interface choices, the homogeneous physical-boundary traces
are
\begin{equation}
\label{eq:homogeneous_aux_boundary_traces}
\begin{aligned}
x=0:\quad
&\widehat{\uoneh}=\zerovec{6},
&\widehat{\Pimat\bm r_h}&=\Pimat\bm r_h^+,
\\
x=\ell:\quad
&\widehat{\uoneh}=\uoneh^-,
&\widehat{\Pimat\bm r_h}&=\zerovec{12}.
\end{aligned}
\end{equation}
We extend the continuous energy norm to the broken finite element space by
\begin{equation}
\label{eq:discrete_energy_norm}
    \engnorm{\uh}^2
    :=
    \sum_{i=1}^{N_c}\inproduct[\Omega_i]{\Gammamat\uh}{\uh},
\end{equation}
and write the physical-boundary dissipation as
\begin{equation}
\label{eq:boundary_dissipation}
B(t):=
(\uoneh^{+})^\top
\bm S_0\,\uoneh^{+}\Big|_{x=0}
+
(\utwoh^{-})^\top
\bm S_\ell\,\utwoh^{-}\Big|_{x=\ell}
\ge 0.
\end{equation}

\pagebreak[3]
\begin{theorem}[Semi-discrete mechanical energy identity and stability]
\label{thm:semidiscrete_energy}
Assume the coefficient hypotheses of Theorem~\ref{thm:continuous_energy}, with
material coefficients continuous across element interfaces and the element
inner products in Problem~\ref{prob:global_formulation} evaluated exactly. Let
\((\uh,\rtauh)\) solve Problem~\ref{prob:global_formulation} with the
homogeneous hyperbolic and auxiliary boundary states
\eqref{eq:homogeneous_outer_states}--\eqref{eq:homogeneous_aux_boundary_traces}.
At interior interfaces, let the characteristic flux~\eqref{eq:flux_pi_u} use
\(\sigma\ge0\), and choose either the alternating LDG traces
\eqref{eq:flux_aux}, the central BR1 traces, or any auxiliary traces satisfying
\eqref{eq:aux_trace_compatibility}. At the physical boundaries, use the fully
upwind characteristic flux \(\sigma=1\). Then, for almost every \(t\in(0,T)\),
\begin{align}
\label{eq:disc_damp_eqn_final}
\frac12 \frac{\mathrm{d}}{\mathrm{d}t} \engnorm{\uh}^2
&+
\sum_{i\in\mathcal I_{\mathrm{int}}}
\interfacepoint[x_i]{\frac{\sigma}{2}\,\Gammamat|\Amat|\,\jump{\uh}}{\jump{\uh}}
+
\sum_{i=1}^{N_c}\inproduct[\Omega_i]{\Hmat^{-1}\rtauh}{\rtauh}
\nonumber\\
&\quad
+B(t)
=
\sum_{i=1}^{N_c}\inproduct[\Omega_i]{\uh}{\qext}.
\end{align}
If additionally
\(\qext\in L^2(0,T;L^2_{\Gammamat^{-1}}(\Omega)^{12})\), then
\begin{align}
\label{eq:dg_final_bound}
\engnorm{\uh(\cdot,T)}^2
&+ 2\int_0^T
\sum_{i\in\mathcal I_{\mathrm{int}}}
\interfacepoint[x_i]{\tfrac{\sigma}{2}\,\Gammamat|\Amat|\,\jump{\uh(\cdot,t)}}{\jump{\uh(\cdot,t)}}\,\mathrm{d}t
+ 2\int_0^T \sum_{i=1}^{N_c}\inproduct[\Omega_i]{\Hmat^{-1}\rtauh(\cdot,t)}{\rtauh(\cdot,t)}\,\mathrm{d}t
\nonumber\\
&\quad
+ 2\int_0^T B(t)\,\mathrm{d}t
\le
e^T\left(
\engnorm{\bm u_{h,0}}^2
+ \int_0^T \|\qext(\cdot,t)\|_{\Gammamat^{-1}}^2\,\mathrm{d}t
\right).
\end{align}
\end{theorem}

\begin{proof}
Choosing $\bm v_h=\uh$ and $\bm w_h=\rtauh$ in
\eqref{eq:global_1_hat}--\eqref{eq:global_2_hat} and adding the results yields
\begin{align}
\label{eq:combined_global}
\frac{1}{2}\frac{\mathrm{d}}{\mathrm{d}t} \engnorm{\uh}^2
+\sum_{i=1}^{N_c}\inproduct[\Omega_i]{\Hmat^{-1}\rtauh}{\rtauh}
=&\ \dlinformA{\uh}{\uh}
+\dlinformB{\bm r_h}{\uh}
+\sum_{i=1}^{N_c}\inproduct[\Omega_i]{\bm Q(\uh,\bm r_h)}{\uh}
\\ \notag
&\ +\dlinformC{\uoneh}{\rtauh}
-\sum_{i=1}^{N_c}\inproduct[\Omega_i]{\Emat^\top \uoneh}{\rtauh}
+\sum_{i=1}^{N_c}\inproduct[\Omega_i]{\Lone^\top(\uoneh)\flexmat\utwoh}{\rtauh},
\end{align}

Similar to the continuous analysis, we examine the contributions of the terms on the
right-hand side of~\eqref{eq:combined_global} separately, starting with the discrete form
$\dlinformA{\uh}{\uh}$ associated with the first-order spatial operator.

\paragraph{Contribution of the first-order spatial operator}
\label{par:first_order_contribution}

Recall that $\Pimat$ is symmetric and independent of $x$.
On each element $\Omega_i$, we have
\begin{equation}
(\Pimat \uh,\uh')_{\Omega_i}
= \int_{\Omega_i} (\Pimat \uh)^\top \uh'\,\mathrm{d}x
= \frac12 \int_{\Omega_i} \bigl(\uh^\top \Pimat \uh\bigr)'\,\mathrm{d}x,
\end{equation}
and hence, by the fundamental theorem of calculus,
\begin{equation}
(\Pimat \uh,\uh')_{\Omega_i}
= \frac12 \Bigl[\uh^\top \Pimat \uh\Bigr]_{x_{i-1}}^{x_i}.
\end{equation}
Summing over all elements yields interface and boundary contributions,
\begin{equation}
\sum_{i=1}^{N_c} (\Pimat \uh,\uh')_{\Omega_i}
=
\frac12 \sum_{i\in\mathcal I_{\mathrm{int}}} \jump{\uh^\top \Pimat \uh}\big|_{x_i}
+\frac12\,\externalboundary{\Pimat \uh}{\uh}.
\end{equation}
Using symmetry of $\Pimat$ and the identity
$\jump{\bm u^\top \Pimat \bm u}=2\,\interfacepoint[x_i]{\Pimat\avg{\bm u}}{\jump{\bm u}}$,
we obtain
\begin{equation}
\sum_{i=1}^{N_c} (\Pimat \uh,\uh')_{\Omega_i}
=
\sum_{i\in\mathcal I_{\mathrm{int}}}
\interfacepoint[x_i]{\Pimat\avg{\uh}}{\jump{\uh}}
+\frac12\,\externalboundary{\Pimat \uh}{\uh}.
\end{equation}

By definition,
\begin{equation}
\dlinformA{\uh}{\uh}
=
\sum_{i=1}^{N_c}(\Pimat\uh,\uh')_{\Omega_i}
-\sum_{i\in\mathcal I_{\mathrm{int}}}\interfacepoint[x_i]{\widehat{\Pimat \uh}}{\jump{\uh}}
-\externalboundary{\widehat{\Pimat \uh}}{\uh}.
\end{equation}
Substituting the previous decomposition gives
\begin{align}
\dlinformA{\uh}{\uh}
&=
-\sum_{i\in\mathcal I_{\mathrm{int}}}
\interfacepoint[x_i]{\widehat{\Pimat \uh}-\Pimat\avg{\uh}}{\jump{\uh}}
-\externalboundary{\widehat{\Pimat \uh}-\tfrac12 \Pimat\uh}{\uh}.
\end{align}

Finally, inserting the numerical flux yields
\begin{equation}
\label{eq:contr_discr_firstorder_form}
\dlinformA{\uh}{\uh}
=
-\sum_{i\in\mathcal I_{\mathrm{int}}}
\interfacepoint[x_i]{\frac{\sigma}{2}\,\Gammamat|\Amat|\,\jump{\uh}}{\jump{\uh}}
-\externalboundary{\widehat{\Pimat \uh}-\tfrac12 \Pimat\uh}{\uh}.
\end{equation}

\paragraph{Contribution of the auxiliary coupling terms}
\label{par:aux_coupling_contribution}

We now analyze the coupling contributions
$\dlinformB{\bm r_h}{\uh}$ and $\dlinformC{\uoneh}{\rtauh}$.
Using $\Pimat \bm r_h=[-\rtauh^\top,\,\zerovec{6}^\top]^\top$ and the
trace-compatibility identity~\eqref{eq:aux_trace_compatibility}, we obtain the
following exact cancellation of all interior contributions for both the
alternating pair~\eqref{eq:flux_aux} and the central BR1 pair.

\paragraph{Rewriting $\dlinformB{\bm r_h}{\uh}$}
By definition,
\begin{equation}
\dlinformB{\bm r_h}{\uh}
=
\sum_{i=1}^{N_c}(\Pimat\bm r_h,\uh')_{\Omega_i}
-\sum_{i\in\mathcal I_{\mathrm{int}}}\interfacepoint[x_i]{\widehat{\Pimat \bm r_h}}{\jump{\uh}}
-\externalboundary{\widehat{\Pimat \bm r_h}}{\uh}.
\end{equation}
Since $\Pimat\bm r_h=[-\rtauh^\top,\,\zerovec{6}^\top]^\top$, we have on each element
\begin{equation}
(\Pimat\bm r_h,\uh')_{\Omega_i}
=-(\rtauh,\uoneh')_{\Omega_i}.
\end{equation}
Moreover, at an interior interface \(x_i\), the auxiliary flux gives
\begin{equation}
\interfacepoint[x_i]{\widehat{\Pimat \bm r_h}}{\jump{\uh}}
=
\interfacepointsingle[x_i]{-\widehat{\rtauh}}{\jump{\uoneh}}.
\end{equation}
Hence,
\begin{equation}
\label{eq:aux_b_rewrite}
\dlinformB{\bm r_h}{\uh}
=
-\sum_{i=1}^{N_c}(\rtauh,\uoneh')_{\Omega_i}
+\sum_{i\in\mathcal I_{\mathrm{int}}}\interfacepointsingle[x_i]{\widehat{\rtauh}}{\jump{\uoneh}}
-\externalboundary{\widehat{\Pimat \bm r_h}}{\uh}.
\end{equation}

\paragraph{Adding $\dlinformC{\uoneh}{\rtauh}$ and cancelling interior terms}
Similarly,
\begin{equation}
\dlinformC{\uoneh}{\rtauh}
=
\sum_{i\in\mathcal I_{\mathrm{int}}}\interfacepointsingle[x_i]{\widehat{\uoneh}}{\jump{\rtauh}}
+\externalboundarysingle{\widehat{\uoneh}}{\rtauh}
-\sum_{i=1}^{N_c}(\uoneh,\rtauh')_{\Omega_i}.
\end{equation}
Combining with \eqref{eq:aux_b_rewrite} yields
\begin{align}
\dlinformB{\bm r_h}{\uh}&+\dlinformC{\uoneh}{\rtauh}
={}
-\sum_{i=1}^{N_c}\Bigl[(\rtauh,\uoneh')_{\Omega_i}+(\uoneh,\rtauh')_{\Omega_i}\Bigr]\\
&+\sum_{i\in\mathcal I_{\mathrm{int}}}\interfacepointsingle[x_i]{\widehat{\rtauh}}{\jump{\uoneh}}
+\sum_{i\in\mathcal I_{\mathrm{int}}}\interfacepointsingle[x_i]{\widehat{\uoneh}}{\jump{\rtauh}}
+\externalboundarysingle{\widehat{\uoneh}}{\rtauh}
-\externalboundary{\widehat{\Pimat \bm r_h}}{\uh}.
\end{align}
On each element $\Omega_i$, integration by parts gives
\begin{equation}
(\rtauh,\uoneh')_{\Omega_i}+(\uoneh,\rtauh')_{\Omega_i}
=\int_{\Omega_i}(\uoneh^\top\rtauh)'\,\mathrm{d}x
=\Bigl[\uoneh^\top\rtauh\Bigr]_{x_{i-1}}^{x_i}.
\end{equation}
Summing over all elements therefore yields
\begin{equation}
-\sum_{i=1}^{N_c}\Bigl[\uoneh^\top\rtauh\Bigr]_{x_{i-1}}^{x_i}
=
-\sum_{i\in\mathcal I_{\mathrm{int}}}\jump{\uoneh^\top\rtauh}\big|_{x_i}
-\externalboundarysingle{\uoneh}{\rtauh}.
\end{equation}
The trace-compatibility identity~\eqref{eq:aux_trace_compatibility} therefore
cancels the interior interface contributions exactly, and we obtain the
boundary-only result
\begin{equation}
\label{eq:contr_discr_coupling_forms}
\dlinformB{\bm r_h}{\uh}+\dlinformC{\uoneh}{\rtauh}
=
\externalboundarysingle{\widehat{\uoneh}}{\rtauh}
-\externalboundary{\widehat{\Pimat \bm r_h}}{\uh}
-\externalboundarysingle{\uoneh}{\rtauh}.
\end{equation}
Thus, exactly as in the continuous analysis, the auxiliary coupling contributes
only boundary terms to the discrete energy balance.

\paragraph{Contribution of the lower-order terms}
\label{par:lower_order_contribution}

The cancellations used in the continuous estimate carry over verbatim on the
semi-discrete level, since the same $L^2$ inner products are employed elementwise.
Proceeding as in the derivation of \eqref{eq:source_term_simplified} (using
\eqref{eq:q_def} and \eqref{eq:source_term_components}), we obtain
\begin{align}
\label{eq:contr_discr_lower_order_terms}
\sum_{i=1}^{N_c}\inproduct[\Omega_i]{\bm Q(\uh,\bm r_h)}{\uh}
=
-\sum_{i=1}^{N_c}\inproduct[\Omega_i]{\Lone^\top(\uoneh)\flexmat \utwoh}{\rtauh}
+\sum_{i=1}^{N_c}\inproduct[\Omega_i]{\Emat^\top \uoneh}{\rtauh}
+\sum_{i=1}^{N_c}\inproduct[\Omega_i]{\qext}{\uh}.
\end{align}

\paragraph{Combined energy equation}
\label{par:combined_energy_equation}

We now combine the contributions of the first-order operator
\eqref{eq:contr_discr_firstorder_form}, the auxiliary coupling
\eqref{eq:contr_discr_coupling_forms}, and the lower-order terms
\eqref{eq:contr_discr_lower_order_terms} in \eqref{eq:combined_global}.
After cancellation of the internal lower-order contributions, we obtain
\begin{align}
\label{eq:energy_contribution_with_fluxes}
\frac12 \frac{\mathrm{d}}{\mathrm{d}t} \engnorm{\uh}^2
&+
\sum_{i\in\mathcal I_{\mathrm{int}}}
\interfacepoint[x_i]{\frac{\sigma}{2}\,\Gammamat|\Amat|\,\jump{\uh}}{\jump{\uh}}
+
\sum_{i=1}^{N_c}\inproduct[\Omega_i]{\Hmat^{-1}\rtauh}{\rtauh}
+
\externalboundary{\widehat{\Pimat \uh}-\tfrac12 \Pimat \uh}{\uh}
\nonumber\\
&=
\sum_{i=1}^{N_c}\inproduct[\Omega_i]{\qext}{\uh}
+
\externalboundarysingle{\widehat{\uoneh}}{\rtauh}
-
\externalboundary{\widehat{\Pimat \bm r_h}}{\uh}
-
\externalboundarysingle{\uoneh}{\rtauh}.
\end{align}

This relation represents the semi-discrete analogue of the continuous energy balance.

\paragraph{Boundary treatment}
\label{par:boundary_treatment}

To close the discrete energy balance, we use the boundary states and traces in
\eqref{eq:homogeneous_outer_states}--\eqref{eq:homogeneous_aux_boundary_traces}.
They make the auxiliary coupling contribute no boundary term and make the
first-order operator produce dissipation.

\paragraph{Auxiliary coupling traces}
The physical-boundary traces are kept fixed for both auxiliary-interface
choices. They coincide with extending the global orientation
\eqref{eq:flux_aux} to the physical boundaries by treating the boundary state
as the missing exterior trace. For homogeneous data they are given by
\eqref{eq:homogeneous_aux_boundary_traces}.
Under the homogeneous-data assumption used in this section, the boundary
contribution in \eqref{eq:contr_discr_coupling_forms} cancels identically at
both ends. For instance, at $x=0$,
\begin{align}
\label{eq:diffusion_flux_bnd_left}
(\widehat{\uoneh})^\top\rtauh
-(\widehat{\Pimat\bm r_h})^\top\uh
-\uoneh^\top\rtauh
=
-(\Pimat\bm r_h^+)^\top\uh^+
-(\uoneh^+)^\top\rtauh^+
=
(\rtauh^+)^\top\uoneh^+-(\uoneh^+)^\top\rtauh^+=0,
\end{align}
and similarly at $x=\ell$.

\paragraph{First-order operator flux at the boundaries}
For the boundary treatment of the first-order spatial operator, we use the
Lax--Friedrichs flux \eqref{eq:flux_pi_u} with fully upwind dissipation
$\sigma=1$ and evaluate $\widehat{\Pimat\uh}$ using the outer states
\eqref{eq:outer_states}. Together with the
contraction identity from \ref{app:boundary_contribution_analysis},
this yields the boundary evaluations
\begin{align}
\label{eq:advection_flux_left}
\bigl(\widehat{\Pimat\uh}-\tfrac12\Pimat\uh\bigr)^\top \uh \Big|_{0}
&=
-\,(\uoneh^{+})^{\!\top}
\bm S_0\,\uoneh^{+}\Big|_{0},
\\
\label{eq:advection_flux_right}
\bigl(\widehat{\Pimat\uh}-\tfrac12\Pimat\uh\bigr)^\top \uh \Big|_{\ell}
&=
(\utwoh^{-})^{\!\top}
\bm S_\ell\,\utwoh^{-}\Big|_{\ell}.
\end{align}
Since the boundary operator
$\externalboundary{\widehat{\Pimat\uh}-\tfrac12\Pimat\uh}{\uh}$ evaluates the
right endpoint minus the left endpoint, \eqref{eq:advection_flux_left} and
\eqref{eq:advection_flux_right} contribute
\((\uoneh^+)^\top\bm S_0\uoneh^++(\utwoh^-)^\top\bm S_\ell\utwoh^-\).
Both terms are nonnegative because $\bm S_0$ and $\bm S_\ell$ are symmetric
positive definite. Consequently, the first-order boundary contribution is
dissipative.

\paragraph{Final energy--dissipation identity}
\label{par:final_energy_identity}

Substituting \eqref{eq:advection_flux_left} and \eqref{eq:advection_flux_right}
into \eqref{eq:energy_contribution_with_fluxes} yields
\eqref{eq:disc_damp_eqn_final}, which proves the energy identity.
Equation~\eqref{eq:disc_damp_eqn_final} is the semi-discrete counterpart of
\eqref{eq:combined_energy_bc_0}; the jump dissipation term reflects the presence of
element-wise discontinuities in the DG approximation.

\paragraph{Positive definiteness of dissipation terms}
\label{par:positivity_terms}

As shown in \ref{app:flux_matrices}, the matrix \(\Gammamat|\Amat|\)
decomposes into \(6\times6\) diagonal blocks,
\begin{align}
\label{eq:gamma_abs_a_block}
\Gammamat|\Amat|
=
\begin{bmatrix}
\bm S_0 & \zeromat{6}{6}\\
\zeromat{6}{6} & \bm S_\ell
\end{bmatrix}.
\end{align}
Since \(\bm S_0\) and \(\bm S_\ell\) are symmetric positive
definite, each diagonal block (and hence \(\Gammamat|\Amat|\)) is symmetric
positive definite. Therefore,
\begin{equation}
\sum_{i\in\mathcal I_{\mathrm{int}}}
\interfacepoint[x_i]{\tfrac{\sigma}{2}\,\Gammamat|\Amat|\,\jump{\uh}}{\jump{\uh}}\ge 0.
\end{equation}
Moreover, since \(\Hmat^{-1}\) is symmetric positive definite,
\begin{equation}
\sum_{i=1}^{N_c}\inproduct[\Omega_i]{\Hmat^{-1}\rtauh}{\rtauh}\ge 0.
\end{equation}
Finally, the boundary quadratic forms in \eqref{eq:disc_damp_eqn_final} are
nonnegative pointwise in time.

\paragraph{Energy bound including boundary traces}
\label{par:energy_bound_traces}

To derive a meaningful bound for non-zero external forces and moments,
we proceed analogously to the continuous case in
Section~\ref{sec:continuous_analysis} and assume that the external
data are bounded in the sense that
\begin{equation}
\int_0^T \|\qext(\cdot,t)\|_{\Gammamat^{-1}}^2\,\mathrm{d}t < \infty.
\end{equation}
By Cauchy--Schwarz in the \(\Gammamat\)-inner product and Young's inequality,
\begin{equation}
\sum_{i=1}^{N_c}\inproduct[\Omega_i]{\uh}{\qext}
\le
\frac12\engnorm{\uh}^2 + \frac12\|\qext\|_{\Gammamat^{-1}}^2.
\end{equation}
Then \eqref{eq:disc_damp_eqn_final} implies, for a.e.\ \(t\in(0,T)\),
\begin{align}
\label{eq:pointwise_bound}
\frac12\frac{\mathrm{d}}{\mathrm{d}t} \engnorm{\uh(\cdot,t)}^2
&+
\sum_{i\in\mathcal I_{\mathrm{int}}}
\interfacepoint[x_i]{\tfrac{\sigma}{2}\,\Gammamat|\Amat|\,\jump{\uh(\cdot,t)}}{\jump{\uh(\cdot,t)}}
+
\sum_{i=1}^{N_c}\inproduct[\Omega_i]{\Hmat^{-1}\rtauh(\cdot,t)}{\rtauh(\cdot,t)}
+ B(t)
\nonumber\\
&\le
\frac12\engnorm{\uh(\cdot,t)}^2
+\frac12\|\qext(\cdot,t)\|_{\Gammamat^{-1}}^2.
\end{align}
Integrating \eqref{eq:pointwise_bound} over \([0,T]\) yields
\begin{align}
\label{eq:integrated_bound}
\engnorm{\uh(\cdot,T)}^2
&+ 2\int_0^T
\sum_{i\in\mathcal I_{\mathrm{int}}}
\interfacepoint[x_i]{\tfrac{\sigma}{2}\,\Gammamat|\Amat|\,\jump{\uh(\cdot,t)}}{\jump{\uh(\cdot,t)}}
\,\mathrm{d}t
+ 2\int_0^T \sum_{i=1}^{N_c}\inproduct[\Omega_i]{\Hmat^{-1}\rtauh(\cdot,t)}{\rtauh(\cdot,t)}\,\mathrm{d}t
\nonumber\\
&\quad
+ 2\int_0^T B(t)\,\mathrm{d}t
\le
\engnorm{\uh(\cdot,0)}^2
+ \int_0^T \engnorm{\uh(\cdot,t)}^2\,\mathrm{d}t
+ \int_0^T \|\qext(\cdot,t)\|_{\Gammamat^{-1}}^2\,\mathrm{d}t.
\end{align}

Define
\begin{align}
\label{eq:zh_def}
\begin{split}
z_h(t) :=\;&
\engnorm{\uh(\cdot,t)}^2
+ 2\int_0^t
\sum_{i\in\mathcal I_{\mathrm{int}}}
\interfacepoint[x_i]{\tfrac{\sigma}{2}\,\Gammamat|\Amat|\,\jump{\uh(\cdot,s)}}{\jump{\uh(\cdot,s)}}\,\mathrm{d}s\\
&\;+
2\int_0^t \sum_{i=1}^{N_c}\inproduct[\Omega_i]{\Hmat^{-1}\rtauh(\cdot,s)}{\rtauh(\cdot,s)}\,\mathrm{d}s
+ 2\int_0^t B(s)\,\mathrm{d}s.
\end{split}
\end{align}
Then \(\int_0^t \engnorm{\uh(\cdot,s)}^2\,\mathrm{d}s
\le \int_0^t z_h(s)\,\mathrm{d}s\),
and \eqref{eq:integrated_bound} implies
\begin{equation}
z_h(T) \le z_h(0) + \int_0^T z_h(t)\,\mathrm{d}t
+ \int_0^T \|\qext(\cdot,t)\|_{\Gammamat^{-1}}^2\,\mathrm{d}t.
\end{equation}
An application of Grönwall's inequality yields
\eqref{eq:dg_final_bound}, completing the proof.
\end{proof}

This estimate bounds the semi-discrete energy and the cumulative interface,
material, and boundary dissipation on $[0,T]$ in terms of the initial state and
external forcing. It is a spatial, semi-discrete result; it does not by itself
establish convergence or an energy law for the time integrator.

\begin{remark}[Inhomogeneous boundary data]
\label{rem:inhomogeneous_bc_data}
The homogeneity of the boundary data is essential for the form
of~\eqref{eq:disc_damp_eqn_final}. Repeating the computation with the general
outer states~\eqref{eq:outer_states} adds the data terms
\begin{equation}
\bm\alpha^\top\bigl(\bm S_0\,\uoneh^{+}-\utwoh^{+}\bigr)\Big|_{x=0}
+\bigl(\bm u_{2,\mathrm{bc}}^\ell\bigr)^\top
\bigl(\uoneh^{-}+\bm S_\ell\,\utwoh^{-}\bigr)\Big|_{x=\ell}
\end{equation}
to the right-hand side, together with analogous contributions from the
data-carrying auxiliary traces
\(\widehat{\uoneh}(0)=\bm\alpha\) and
\(\widehat{\Pimat\bm r_h}(\ell)
=\Pimat[\zerovec{6}^\top,((\rtau)_{\mathrm{bc}}^\ell)^\top]^\top\).
These terms are linear in the data but couple to boundary traces of the
discrete solution. They are sign-indefinite and cannot be absorbed into the
available dissipative terms by completing the square, which is the same
obstruction encountered for the undamped intrinsic beam equations
in~\cite{Bleffert2025}. Consequently, the present calculation does not establish
an energy estimate for this inhomogeneous closure; such an estimate is deferred
to future work (cf.\ Section~\ref{sec:future_work}).
\end{remark}

In the next section, we assess implementation accuracy, observed stability,
and convergence in numerical experiments.

\section{Numerical Results}
\label{sec:numerical_analysis}
Having established discrete energy stability for the damped intrinsic beam
system (Problem~\ref{prob:initial_boundary_value}), we now verify the implementation
and assess the convergence behavior of the semi-discrete DG scheme within the simulation
framework \trixi~\cite{schlottke2021trixi,Ranocha2022}. The corresponding
LGL/SBP DGSEM realization is detailed in~\ref{app:dgsem}.
The source code and numerical data supporting the results are available from
the corresponding author upon reasonable request.

\subsection{Convergence Tests via the Method of Manufactured Solutions}
\label{subsec:mms}

To validate correctness and quantify convergence, we employ the \emph{method of manufactured
solutions} (MMS), which enables a systematic computation of the \emph{experimental order of
convergence} (EOC) by prescribing an exact solution and deriving the corresponding forcing
terms; see, e.g.,~\cite{Roache2002}.

\paragraph{Manufactured solution}
We use a straight reference beam, $\bm\kappa_0=\bm0$, on
$\Omega=[0,\ell]$ with $\ell=1$. Let
\begin{equation}
 \bm e:=(1,0,0,0,0,0)^\top,
 \qquad
 \bm q:=(3,2,3,1,1,1)^\top,
 \qquad
 \psi(x,t):=\exp(x+t).
\end{equation}
For this straight configuration, $\Emat=\Lone(\bm e)$ and
$(\Emat^\top)^2=0$. We manufacture
\begin{equation}
\label{eq:mms_physical_solution}
\bm u_{1,\mathrm{ex}}(x,t)
 =(\idmat{6}{6}+t\Emat^\top)\psi(x,t)\bm q,
\qquad
\flexmat\bm u_{2,\mathrm{ex}}(x,t)
 =\bm u_{1,\mathrm{ex}}(x,t)-2\bm e,
\qquad
\rtauex
 =\dampmat\dot{\bm u}_{2,\mathrm{ex}}.
\end{equation}
This elementary invariant construction satisfies the unforced lower six
equations rather than prescribing the two state blocks independently. The
intrinsic identities
\begin{equation}
\Lone^\top(\bm a)\bm a=\bm0,
\qquad
\Lone^\top(\bm a)\bm e=-\Emat^\top\bm a
\end{equation}
give the nonzero relation
\begin{equation}
\Lone^\top(\bm u_{1,\mathrm{ex}})
 \flexmat\bm u_{2,\mathrm{ex}}
=2\Emat^\top\bm u_{1,\mathrm{ex}}.
\end{equation}
Moreover, direct differentiation of~\eqref{eq:mms_physical_solution} gives
$\dot{\bm u}_{1,\mathrm{ex}}=(\bm u_{1,\mathrm{ex}})'
+\Emat^\top\bm u_{1,\mathrm{ex}}$. Consequently,
\begin{align}
\flexmat\dot{\bm u}_{2,\mathrm{ex}}
&=\dot{\bm u}_{1,\mathrm{ex}} \nonumber\\
&=(\bm u_{1,\mathrm{ex}})'
  -\Emat^\top\bm u_{1,\mathrm{ex}}
  +\Lone^\top(\bm u_{1,\mathrm{ex}})
   \flexmat\bm u_{2,\mathrm{ex}},
\label{eq:mms_zero_lower_source}
\end{align}
and, with $\rtauex=\dampmat\dot{\bm u}_{2,\mathrm{ex}}$,
the auxiliary relation is also satisfied identically:
\begin{equation}
\label{eq:mms_rtau_from_u}
\Hmat^{-1}\rtauex
=
(\bm u_{1,\mathrm{ex}})'
-\Emat^\top \bm u_{1,\mathrm{ex}}
+\Lone^\top(\bm u_{1,\mathrm{ex}})
 \flexmat \bm u_{2,\mathrm{ex}}.
\end{equation}
Thus $\rtauex$ is simultaneously the Kelvin--Voigt resultant
$\dampmat\dot{\bm u}_{2,\mathrm{ex}}$ and the field obtained from the
auxiliary equation; these are not two independent prescriptions. All twelve
primary components and all six viscous-resultant components are pointwise
active on the test domain. Both compatibility terms are nonzero, but their
displayed relation makes clear that this single MMS does not excite their
coefficients independently; a targeted member with $\lambda=3/2$ is retained
as a coefficient guard in the reproducibility tests.

\paragraph{Manufactured forcing}
Let $\bm Q_0$ denote the source in~\eqref{eq:q_def} without the external-load
term $\qext$.  We manufacture only the physical six-component vector of
distributed forces and moments:
\begin{equation}
\label{eq:mms_forcing_physical}
\bm f_{e,\mathrm{MMS}}
:=
\begin{bmatrix}\idmat{6}{6}&\zeromat{6}{6}\end{bmatrix}
\left(
\Gammamat\dot{\bm u}_{\mathrm{ex}}
+\Pimat\bm u_{\mathrm{ex}}'
+\Pimat\bm r_{\mathrm{ex}}'
-\bm Q_0(\bm u_{\mathrm{ex}},\bm r_{\mathrm{ex}})
\right),
\qquad
\bm q_{\mathrm{ext,MMS}}
:=
\begin{bmatrix}
\bm f_{e,\mathrm{MMS}}\\ \zerovec{6}
\end{bmatrix},
\end{equation}
where
$\bm r_{\mathrm{ex}}=[\zerovec{6}^\top,\rtauex^\top]^\top$.
Equation~\eqref{eq:mms_zero_lower_source} proves that the lower block of the
residual in parentheses is exactly zero.  No manufactured compatibility source
is introduced.  In the implementation this property is checked before every
convergence run; the MMS driver aborts if the lower-block residual exceeds
roundoff tolerance.

\paragraph{Boundary and initial data for MMS}
We prescribe the exact inhomogeneous split targets
\begin{equation}
\bm u_{1,\mathrm{bc}}^0(t)=\bm u_{1,\mathrm{ex}}(0,t),
\qquad
\bm u_{2,\mathrm{bc}}^\ell(t)=\bm u_{2,\mathrm{ex}}(\ell,t),
\qquad
\bm r_{\tau,\mathrm{bc}}^\ell(t)=\rtauex(\ell,t).
\end{equation}
Their sum is the prescribed total right mechanical resultant. The convergence
test therefore exercises the same nonhomogeneous left hyperbolic/gradient and
right hyperbolic/divergence boundary functions used by the implementation; it
is a code-verification problem, not an instance of the homogeneous-data energy
theorem. The initial state is $\bm u_{\mathrm{ex}}(x,0)$ from
\eqref{eq:mms_physical_solution}.

\paragraph{Material matrices and discretization parameters}
The compliance, mass, and damping matrices used in these tests are
\setlength{\arraycolsep}{2pt}
\renewcommand{\arraystretch}{0.85}
\begin{align}
\flexmat &=
\begin{pmatrix}
1.0 & 0.3 & 0.3 & 0.12 & 0.05 & 0.025 \\
0.3 & 2.0 & 0.47 & 0.2  & 0.12 & 0.05  \\
0.3 & 0.47& 3.0  & 0.35 & 0.2  & 0.12  \\
0.12& 0.2 & 0.35 & 4.0  & 0.34 & 0.2   \\
0.05& 0.12& 0.2  & 0.34 & 5.0  & 0.34  \\
0.025&0.05&0.12 & 0.2  & 0.34 & 6.0
\end{pmatrix},
&
\massmat &=
\begin{pmatrix}
1.84147 & 0.25403 & 0.34207 & 0.11081 & 0.05841 & 0.02270 \\
0.25403 & 2.84147 & 0.43782 & 0.22524 & 0.11081 & 0.05841 \\
0.34207 & 0.43782 & 3.84147 & 0.32702 & 0.22524 & 0.11081 \\
0.11081 & 0.22524 & 0.32702 & 4.84147 & 0.32161 & 0.22524 \\
0.05841 & 0.11081 & 0.22524 & 0.32161 & 5.84147 & 0.32161 \\
0.02270 & 0.05841 & 0.11081 & 0.22524 & 0.32161 & 6.84147
\end{pmatrix},
\\[4pt]
\dampmat &= 10^{-2}\flexmat.
\end{align}
The compliance and mass matrices are symmetric positive definite representative
(though non-physical) test data.  Proportional damping gives
$\Hmat=10^{-2}\idmat{6}{6}$ and reduces explicit parabolic stiffness while
retaining a nonzero Kelvin--Voigt contribution.

Convergence tests were performed for polynomial degrees $k=1,2,3$ on uniform meshes with
$N_c=4,8,16,32,64,128$ elements. Errors were evaluated at final time $T=1$.
The convergence campaign considers $k=1,2,3$.
Time integration uses the adaptive fourth-order stabilized ROCK4 method with
absolute and relative tolerances $10^{-12}$ in every standard convergence run.

We report $L^2(\Omega)$-errors at the final time $T=1$ for the $12$-component state $\bm u=(\bm u_1,\bm u_2)$. To present the results compactly, we plot errors aggregated over the two $6$-component blocks corresponding to the kinematic variables $\bm u_1$ and the sectional resultants $\bm u_2$.
To ensure that these averages do not conceal a poorly converging field, a
companion audit records \(L^2\)- and \(L^\infty\)-errors for all 12
components and for all six reconstructed Kelvin--Voigt resultants. For the
alternating cubic scheme, a matched tolerance-$10^{-13}$ fine-grid audit gives
$L^2$ EOCs in $[3.81,4.07]$ over the 12 primary components and in
$[3.81,4.02]$ over the six viscous-resultant components. On the finest mesh,
tightening the tolerance changes the two displayed block-average errors from
$(2.067863,\,3.689700)\times10^{-11}$ to
$(2.064336,\,3.689580)\times10^{-11}$. Thus neither component averaging nor
temporal error explains the reported primary cubic rates.


\begin{figure}
\centering
\includegraphics[width=0.95\linewidth]{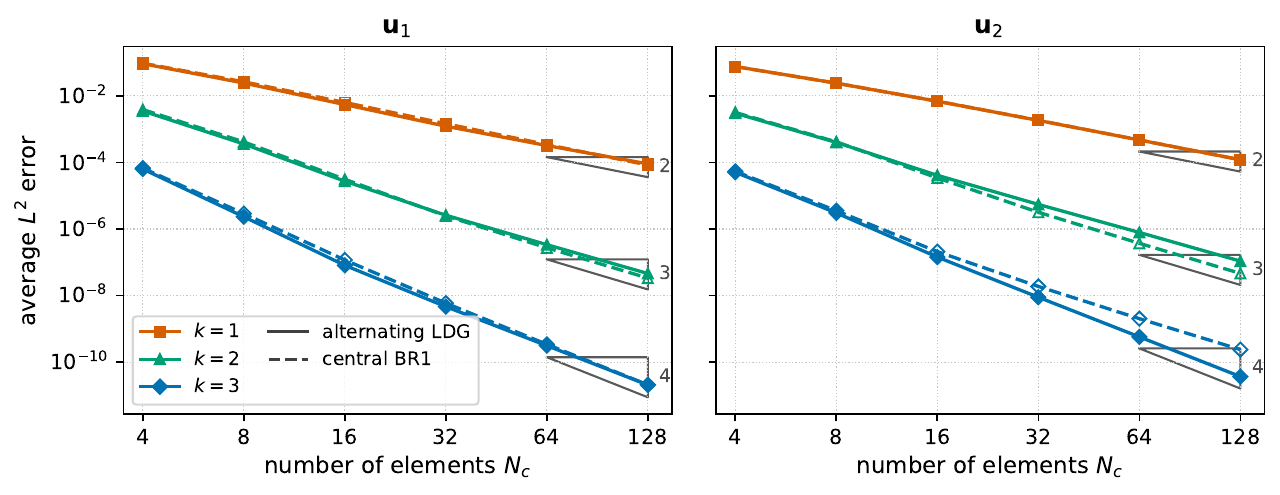}
\caption{Average $L^2$-errors at $T=1$ (arithmetic mean of the componentwise $L^2(\Omega)$ errors
within each $6$-component block) versus the number of elements $N_c$, for polynomial degrees
$k=1,2,3$: kinematic block $\bm u_1$ (left) and resultant block $\bm u_2$ (right).
The characteristic first-order flux is fully upwind (\(\sigma=1\)) in every
run. Solid lines use the alternating auxiliary traces~\eqref{eq:flux_aux};
dashed lines use central BR1 auxiliary traces as a control. Gray triangles
indicate the optimal slopes $k+1$.}
\label{fig:l2_error_plots}
\end{figure}

\paragraph{Discussion of results}
In Figure~\ref{fig:l2_error_plots}, the average $L^2$-errors resulting from our \trixi\, simulations
are depicted; the underlying error and EOC data are tabulated in~\ref{app:tabulated_convergence}.
The comparison changes only the Kelvin--Voigt auxiliary traces: the
manufactured fields, physical forcing, meshes, characteristic-upwind
first-order flux, time integrator, and tolerances are identical.  It therefore
isolates the auxiliary discretization rather than conflating it with the
central-versus-upwind choice for \(\widehat{\Pimat\uh}\).

For \(k=3\), alternating LDG restores near-fourth-order convergence in both
blocks. On the finest mesh, the EOCs are \(3.92\) for \(\bm u_1\) and
\(3.96\) for \(\bm u_2\). In contrast, BR1 retains fourth order in
\(\bm u_1\) but its \(\bm u_2\) EOC decreases to \(3.09\). Because the
comparison changes only the auxiliary closure, it attributes the observed
reduction in this MMS to the central auxiliary traces rather than to
characteristic upwinding or to the manufactured solution. For \(k=2\), the
alternating EOCs are \(2.88\) and \(2.86\), while the BR1 block EOCs are
\(3.02\) and \(3.00\). The componentwise audit nevertheless finds a slower
BR1 $m_2$ field: its finest-pair EOC is $2.60$ and rises to $2.78$ in a
targeted $N_c=256$ extension, so componentwise third order is not claimed for
that comparator. For \(k=1\), both closures give nearly second order in
\(\bm u_2\); the alternating \(\bm u_1\) EOC is \(1.83\) on
the finest pair, so the alternating choice should not be described as a
uniform accuracy improvement in every degree and block.

A rigorous optimal-order \textit{a priori} error analysis for the full coupled
nonlinear system is not yet available (cf.\ Section~\ref{sec:future_work}).
In particular, an alternating-flux proof requires one-sided Gauss--Radau
projections matched to the selected traces rather than the plain elementwise
\(L^2\) projection used for central auxiliary fluxes.  The rates reported
here are therefore experimental.

\subsection{Rotational Beam Experiment}  
\label{subsec:rotational_beam}
We next consider a transient spin-up experiment for a beam that is initially at rest and rotated through a prescribed angular velocity at its root. The imposed angular velocity is ramped from zero to a terminal value $\omega_{\mathrm{ter}}$ over a finite time interval and then held constant. This experiment serves two purposes: it probes the transient response of the full damped intrinsic beam system under rotational loading, and it provides a setting in which the long-time behavior can be compared with an analytic steady reference state. The setup is relevant, for example, to rotor-blade-type dynamics, where centrifugal effects induce axial stretching and damping drives relaxation toward an equilibrium configuration.
We visualize the beam's reference line, reconstructed from the intrinsic variables in a
post-processing step following~\cite{PatilHodges2006}.

Once the imposed angular velocity has reached the constant value $\omega_{\mathrm{ter}}$, the Kelvin--Voigt damped beam enters a terminal constant-spin regime and, in the simulations below, relaxes toward a time-independent state. For the straight beam configuration considered here, this terminal regime is described by a straight, axially stretched rotating branch. We therefore derive an analytic steady solution for this reduced branch and then compare it with numerical solutions of the full system obtained with \trixi\, from rest after the spin-up phase.
Here ``steady'' refers to the intrinsic variables and reconstructed geometry in
the root-attached, co-rotating frame; the corresponding inertial-frame
centerline continues to rotate.

\subsubsection{Setup of the Rotational Beam}

Consider a beam of length $\ell$, attached at $x=0$ and free at $x=\ell$. The root at $x=0$ is fixed in position and driven by a prescribed rotation about the $x_3$-axis, while the tip at $x=\ell$ is traction free.

To obtain a closed analytic reference for the terminal constant-spin regime, we restrict attention to the straight steady branch in which no bending or shear deformation is present and centrifugal loading produces only axial elongation. In this reduced configuration, only the components $f_1$, $v_2$, and $\omega_3$ are nonzero, while
\begin{equation}
f_2 \equiv f_3 \equiv m_1 \equiv m_2 \equiv m_3 \equiv v_1 \equiv v_3 \equiv \omega_1 \equiv \omega_2 \equiv 0.
\end{equation}
This reduction is used only to characterize the terminal steady state; the time-dependent simulations below solve the full damped intrinsic beam system without imposing this restriction.

With $\dot{\bm u}=0$, the reduced steady system reads
\begin{equation}
\label{eq:rot_beam_reduced_system}
\begin{aligned}
    f_1'(x) = -\,\massmat_{22}\,\omega_3\,v_2(x),
    \qquad
    v_2'(x) = \omega_3 + \flexmat_{11}\,\omega_3\,f_1(x),
    \qquad
    \omega_3'(x) &= 0,
\end{aligned}
\end{equation}
supplemented with boundary conditions
\begin{equation}
\label{eq:rot_beam_bc}
    f_1(\ell)=0, \qquad v_2(0)=0,\qquad \omega_3(0)=\omega_{\mathrm{ter}},
\end{equation}
so that $\omega_3\equiv \omega_{\mathrm{ter}}$ in the steady regime.

\subsubsection{Analytic Steady Solution and Critical Speed}
\label{subsubsec:rotational_beam_analytic}

We define the following coefficients:
\begin{equation}
a := \omega_{\mathrm{ter}}\sqrt{\massmat_{22}\flexmat_{11}},
\qquad
b := \sqrt{\frac{\flexmat_{11}}{\massmat_{22}}},
\qquad
c := \frac{1}{\flexmat_{11}\cos(a\ell)}.
\end{equation}
Then the system of ordinary differential equations in space, \eqref{eq:rot_beam_reduced_system}, with boundary conditions \eqref{eq:rot_beam_bc} admits the steady solution
\begin{equation}
f_1(x)=c\,\cos(a x)-\frac{1}{\flexmat_{11}},
\qquad
v_2(x)=c\,b\,\sin(a x),
\qquad
\omega_3\equiv\omega_{\mathrm{ter}}.
\end{equation}

Note that the factor $c$ becomes unbounded as $a\ell\to \pi/2$, so this reduced steady branch ceases to provide a regular reference state as the spin approaches its first singular threshold. 
We therefore introduce the (first) critical speed
\begin{equation}
\omega_{\mathrm{crit}}
:= \frac{\pi}{2\ell\sqrt{\massmat_{22}\flexmat_{11}}},
\end{equation}
so that $c\to\infty$ as $\omega_{\mathrm{ter}}\to\omega_{\mathrm{crit}}$.
We restrict ourselves to $\omega_{\mathrm{ter}}\ll \omega_{\mathrm{crit}}$ in the numerical experiments.

\begin{remark}
Using the definitions of \(\massmat_{22}\) and \(\flexmat_{11}\), one obtains
\begin{equation}
\frac{1}{\sqrt{\massmat_{22}\flexmat_{11}}}=\sqrt{\frac{E}{\rho}},
\end{equation}
where \(E\) is Young's modulus and \(\rho\) the density \cite{Kinsler2000}.
Writing \(c_e=\sqrt{E/\rho}\), the threshold satisfies
\(\omega_{\mathrm{crit}}\ell=(\pi/2)c_e\). Thus it is the dimensionless
rotation parameter \(\omega_{\mathrm{ter}}\ell/c_e\), not equality of the tip
and wave speeds, that reaches \(\pi/2\) when
\(\cos(a\ell)\to0\).
\end{remark}

A simulation initialized at rest, accelerated through the ramp, and then held at $\omega_{\mathrm{ter}}$ should converge to the analytic steady branch during the constant-angular-speed phase; this is assessed below.

\subsubsection{Numerical Simulation Protocol}
\label{subsubsec:rotational_beam_numerics}

We simulate a spin-up from rest by prescribing a time-dependent angular velocity at the left boundary, ramped linearly to $\omega_{\mathrm{ter}}$ over one second and held constant thereafter:
\begin{equation}
\omega_3(0, t)=
\begin{cases}
\omega_{\mathrm{ter}}\,t, & 0\le t\le 1,\\
\omega_{\mathrm{ter}}, & t>1.
\end{cases}
\end{equation}
The material matrices are chosen as
\begin{align}
    \flexmat=\bigl(\operatorname{diag}(10^{3},10^{3},&\,10^{3},500,500,500)\bigr)^{-1},
    \hspace{10ex}
    \massmat=\operatorname{diag}(1,1,1,20,10,10),
    \\[2ex]
    &\dampmat=\bigl(\operatorname{diag}(10^{3},10^{3},10^{3},10,10,10)\bigr)^{-1},
\end{align}
based on parameters from~\cite{HessePalacios2012,SimoVuQuoc1986} (with the stiffness modified according to~\cite{HessePalacios2012}).
For a beam with these material parameters, the critical angular speed is $\omega_{\mathrm{crit}}\approx12.4182$.
As the terminal angular speed, we choose $\omega_{\mathrm{ter}}=2.8523$, corresponding to
$a\ell=\pi/2-1.21$, well below $\omega_{\mathrm{crit}}$.
All transient runs use $\ell=4$, polynomial degree $k=3$, $N_c=8$
uniform cells, characteristic upwinding, the alternating auxiliary traces
in~\eqref{eq:flux_aux}, and the five-stage fourth-order
Carpenter--Kennedy method at CFL \(0.01\).  The baseline, doubled-damping, and
undamped comparison runs use \(T=40\); the consistency run initialized from
the analytic steady profile uses \(T=10\).  There are no distributed loads,
the initial curvature is zero, and the transient runs start from rest.

\subsubsection{Beam Dynamics and Results}
\label{subsubsec:rotational_beam_results}

During the ramp, boundary actuation injects energy and the free end lags the
root.  After the angular speed becomes constant, the beam undergoes damped
oscillations and approaches the straight, axially stretched branch derived
above.  For the baseline run, the analytic steady energy is
\(E_\star=264.3321651\).  The discrete energy rises from zero to
\(253.7525\) at the end of the ramp, reaches \(604.0835\) at \(t=1.7\),
and subsequently relaxes toward \(E_\star\); see
Figure~\ref{fig:rotating_validation}.

\begin{table}[t]
\centering
\caption{Quantitative rotating-branch validation for the \trixi\
baseline run at \(T=40\).  Relative \(L^\infty\) errors are normalized by
the maximum magnitude of the corresponding analytic field.}
\label{tab:rotating_validation}
\begin{tabular}{lc}
\toprule
Diagnostic & Value \\
\midrule
Relative \(L^\infty\)-error in \(f_1\)
& \(2.38\times10^{-5}\) \\
Relative \(L^\infty\)-error in \(v_2\)
& \(1.51\times10^{-5}\) \\
Relative energy error \(\lvert E_h-E_\star\rvert/E_\star\)
& \(2.74\times10^{-5}\) \\
Maximum relative semi-discrete ledger defect
& \(1.78\times10^{-15}\) \\
\bottomrule
\end{tabular}
\end{table}

\begin{figure}[t]
\centering
\includegraphics[width=\textwidth]{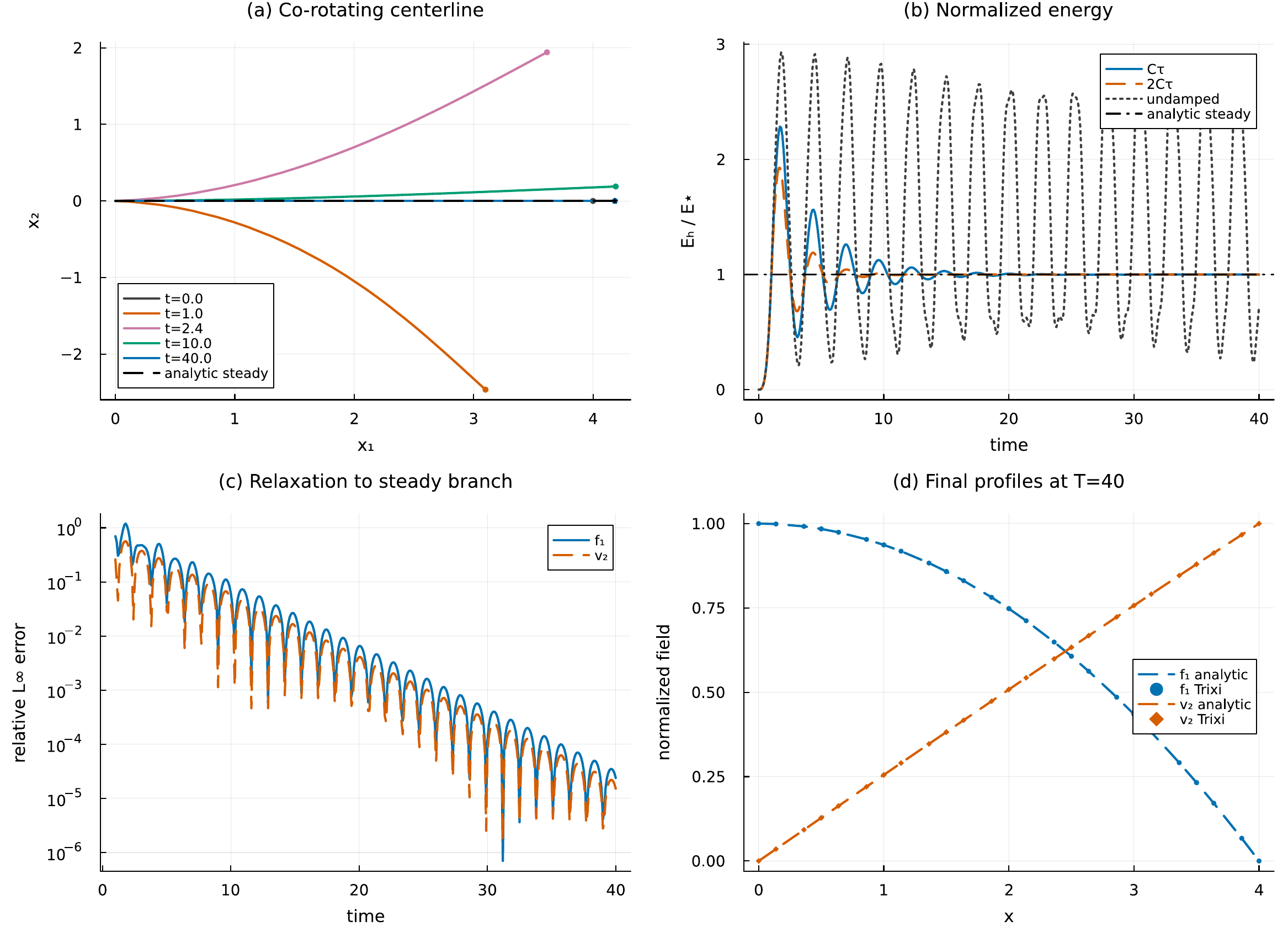}
\caption{\trixi\ rotating-beam validation.
(a) Co-rotating centerlines at the initial state, ramp end, first positive
post-ramp rebound, a late relaxation time, and final time; circles mark the
numerical tips and the star marks the analytic steady tip.  The vertical scale
is expanded to expose the transient lag.
(b) Discrete energy normalized by \(E_\star\) for the baseline,
doubled-damping, and singular undamped comparison.
(c) Relative \(L^\infty\)-errors against the analytic terminal branch after
the ramp.
(d) Normalized numerical and analytic profiles of \(f_1\) and \(v_2\) at
\(T=40\).}
\label{fig:rotating_validation}
\end{figure}

The final-time errors in Table~\ref{tab:rotating_validation} are not used in
isolation: over \(35\le t\le40\), the relative \(L^\infty\)-errors remain
below \(1.00\times10^{-4}\) in \(f_1\) and \(6.27\times10^{-5}\) in
\(v_2\).  Thus the agreement is not an artifact of sampling a single zero
crossing of the decaying oscillation.  Panel~(d) of
Figure~\ref{fig:rotating_validation} compares the complete final profiles
with the analytic branch.

\paragraph{Actuated energy ledger}
A constant prescribed angular velocity can continue to exchange work with
the beam, so this experiment is outside the homogeneous-boundary estimate of
Section~\ref{subsec:discrete_energy_stability}.  For the implemented outer
state, we evaluate the actuated semi-discrete identity in the form
\begin{equation}
\dot E_h+D_{\tau,h}+D_{\mathrm{jump},h}+D_{0,h}+D_{\ell,h}
=P_{\mathrm{root}}+P_{\mathrm{SAT}},
\end{equation}
where
\begin{equation}
P_{\mathrm{root}}
=-\bm\alpha^\top(\utwoh+\rtauh)\big|_{x=0},
\qquad
P_{\mathrm{SAT}}
=\bm\alpha^\top\bm S_0\uoneh\big|_{x=0}.
\end{equation}
Here \(P_{\mathrm{root}}\) is the physical actuator power, whereas
\(P_{\mathrm{SAT}}\) is the characteristic boundary-data contribution and
must not be interpreted as additional physical work.  Evaluating this
identity at every saved state gives the maximum relative defect reported in
Table~\ref{tab:rotating_validation}.  This checks the semi-discrete operator
at the computed states; it is not a fully discrete energy theorem for the
explicit Runge--Kutta method.

As a complementary fully discrete diagnostic, define the accumulated work and
dissipation channels by
\begin{equation}
W_{\mathrm{root}}:=\int_0^T P_{\mathrm{root}}\,\mathrm dt,
\qquad
W_{\mathrm{SAT}}:=\int_0^T P_{\mathrm{SAT}}\,\mathrm dt,
\qquad
\mathcal D_{j,h}:=\int_0^T D_{j,h}\,\mathrm dt,
\quad j\in\{\tau,\mathrm{jump},0,\ell\}.
\end{equation}
Trapezoidal accumulation at every accepted step gives root work
\(W_{\mathrm{root}}=377.935824\), partitioned into
\(\Delta E_h=264.324932\), material dissipation
\(\mathcal D_{\tau,h}=113.610851\), interface dissipation
\(\mathcal D_{\mathrm{jump},h}=1.07\times10^{-6}\), and the net characteristic
boundary correction
\(\mathcal D_{0,h}-W_{\mathrm{SAT}}=4.02\times10^{-5}\) (the free-tip term is
\(\mathcal D_{\ell,h}=2.91\times10^{-8}\)).  The resulting cumulative closure
residual has magnitude \(1.10\times10^{-7}\).  With the accumulated ledger
scale
\begin{equation}
S_{\mathrm{led}}
:=\max\!\left\{
1,\lvert\Delta E_h\rvert,
\sum_{j\in\{\tau,\mathrm{jump},0,\ell\}}\lvert\mathcal D_{j,h}\rvert
+\lvert W_{\mathrm{root}}\rvert+\lvert W_{\mathrm{SAT}}\rvert
\right\},
\end{equation}
the relative residual is \(2.40\times10^{-12}\); relative to the physical root
work alone, it is \(2.91\times10^{-10}\).  Unlike the pointwise defect, this
residual also contains time-integration and temporal-quadrature errors.

Doubling \(\dampmat\) shortens the oscillatory transient while preserving the
same analytic equilibrium.  With \(\dampmat=\zeromat{6}{6}\), the oscillations persist over
the reported interval.  The latter is a singular undamped comparison and is
not covered by the damped mixed estimate containing \(\Hmat^{-1}\).
Finally, a run initialized from the analytic profile at the terminal angular
speed remains at that branch up to discretization and time-integration error:
at \(T=10\), the relative \(L^\infty\)-errors are \(9.76\times10^{-8}\) in
\(f_1\) and \(6.91\times10^{-9}\) in \(v_2\), while the relative energy error
is \(4.59\times10^{-12}\).
An independent steady-initialized mesh check evolved to \(T=1\) on
\(N_c=4,8,16\) gives finest-pair \(L^\infty\)-EOCs of \(3.99\) for
\(f_1\) and \(3.73\) for \(v_2\).  At \(N_c=16\), the relative
\(L^\infty\)-errors in \(f_1\) and \(v_2\) are \(6.15\times10^{-9}\) and
\(5.21\times10^{-10}\), respectively.  The maximum relative instantaneous
residual of the semi-discrete energy identity is at roundoff on every mesh.

\FloatBarrier

\section{Conclusion}
\label{sec:conclusion}

In this work, we developed and analyzed an energy-stable discontinuous
Galerkin method for the Kelvin--Voigt damped intrinsic beam equations.
The main difficulty is the mixed space--time derivative introduced by the
strain-rate-dependent constitutive law, which does not fit directly into
standard hyperbolic or parabolic DG frameworks. To address this issue, we
introduced the auxiliary variable
\begin{equation}
    \rtau := \dampmat\,\dot{\utwo},
\end{equation}
which represents the rate-dependent contribution to the sectional force
and moment resultants. This yields a first-order-in-time mixed formulation
in which the hyperbolic part retains the structure of the undamped
intrinsic beam equations, while the damping enters through an algebraic
auxiliary relation.

For the reformulated system, we derived a mechanical energy--dissipation
identity and an \textit{a priori} stability bound under bounded external loads. The
semi-discrete DG method was designed to mimic this structure at the
discrete level. The characteristic upwind flux for the first-order spatial
operator provides interior jump dissipation, while complementary alternating
LDG traces for the auxiliary coupling cancel the corresponding inter-element
terms. For the stated homogeneous split cantilever closure, the characteristic
outer-state construction gives non-negative boundary contributions. The resulting discrete estimate controls the discrete
energy, the Kelvin--Voigt dissipation, the interior jump dissipation, and
the boundary dissipation in terms of the initial data and the external
forcing.

The numerical experiments support the theoretical findings. For the cubic
scheme with characteristic upwinding and alternating LDG traces, componentwise and
tighter-tolerance checks support near-fourth-order convergence of all twelve
primary variables. A controlled BR1 comparison shows that alternating
auxiliary traces remove the observed cubic resultant-block order loss.
The same comparison also shows that alternating traces are not
uniformly more accurate for every degree and block. In the rotational spin-up
experiment, the full nonlinear damped simulation relaxes toward the analytic
constant-spin branch while closing the actuated energy ledger; a separate
steady-initialized check establishes mesh convergence, and the undamped case
retains persistent oscillations.
Increasing the damping strength accelerates this relaxation, in agreement
with the dissipative structure of the model and the discretization.

Overall, the auxiliary-variable reformulation provides a convenient way
to treat Kelvin--Voigt damping in the intrinsic beam equations while
preserving the underlying energy structure. The resulting DG method
combines high-order accuracy, element locality, and a discrete energy
estimate, making it a suitable foundation for further developments toward
more realistic beam and rotor-blade simulations.

\section{Future Work}
\label{sec:future_work}

Several extensions remain for future work.

\paragraph{Analysis}
The continuous and semi-discrete stability results in
Sections~\ref{sec:continuous_analysis} and~\ref{subsec:discrete_energy_stability}
were derived under homogeneous boundary data. Extending the
analysis to the inhomogeneous velocity and sectional-resultant boundary
data would give a more comprehensive model; the obstruction is the
sign-indefinite boundary work identified in
Remark~\ref{rem:inhomogeneous_bc_data}. A fully discrete
energy analysis for practical time integrators is another important
direction. This includes explicit Runge--Kutta schemes for lightly damped
problems, IMEX methods when the Kelvin--Voigt term introduces stiffness,
and implicit schemes for strongly damped configurations. A complete
\textit{a priori} error analysis for the mixed hyperbolic--parabolic formulation
also remains open.  For the alternating pair, the natural route is a
flux-matched one-sided Gauss--Radau projection that annihilates the expensive
auxiliary face residuals; the LDG framework reviewed in~\cite{XuShu2010} and
the strongly-damped-wave-equation analysis~\cite{LarssonThomeeWahlbin1991}
are natural starting points.

\paragraph{Algorithms}
The present implementation uses uniform meshes and a fixed explicit
time-stepping strategy. Natural algorithmic extensions include
\(hp\)-adaptive refinement, locally implicit or IMEX time integration, and
more systematic treatment of stiffness arising from strong damping or
high-order spatial discretizations. Since the method is formulated in a
local DG framework and implemented in \trixi, it is also well suited for
future performance-oriented developments such as GPU acceleration and
distributed-memory parallelism.

\paragraph{Modelling and applications}
The numerical experiments considered representative beam configurations
designed to assess convergence, the analytic steady branch, and the energy balance.
Future studies should consider more realistic rotor-blade models with
anisotropic composite cross-sections, initial twist, pre-curvature, and
spatially varying material parameters. Comparisons with high-fidelity
three-dimensional finite-element simulations would provide a useful
validation benchmark. Another important direction is coupling the
structural solver to aerodynamic loads, control inputs, and multi-component
assemblies such as hubs or joined beam systems. The energy-stable
structure developed here provides a controlled starting point for such
coupled aeroelastic simulations.

\section*{CRediT authorship contribution statement}
\textbf{Shivam Sundriyal:} Conceptualization; Methodology; Software;
Validation; Visualization; Writing -- original draft; Writing -- review \&
editing; Investigation; Formal analysis.
\textbf{Christian Bleffert:} Conceptualization; Methodology; Supervision;
Writing -- review \& editing.
\textbf{Lukas Dreyer:} Conceptualization; Supervision; Methodology.
\textbf{Melven R\"ohrig-Z\"ollner:} Conceptualization; Supervision;
Writing -- review \& editing.
\textbf{Gregor Gassner:} Writing -- review \& editing; Supervision;
Conceptualization.
\textbf{Vadym Aizinger:} Conceptualization; Supervision;
Writing -- review \& editing; Methodology.

\section*{Declaration of competing interest}
The authors declare no competing financial interests or personal relationships
that could have influenced this work.


\appendix
\renewcommand{\theglobaltheorem}{\Alph{section}.\arabic{globaltheorem}}

\section{Sectional Material Matrices}
\label{app:matrices}

This appendix records the physical meaning and common block structure of the
sectional matrices used in the intrinsic beam model. The notation follows the
ordering used throughout the paper,
\begin{equation}
    \uone=\begin{bmatrix}\bm v\\ \bm\omega\end{bmatrix},
    \qquad
    \begin{bmatrix}\bm f\\ \bm m\end{bmatrix}
    =
    \begin{bmatrix}\text{sectional force}\\ \text{sectional moment}\end{bmatrix},
    \qquad
    \begin{bmatrix}\bm\gamma\\ \bm\kappa\end{bmatrix}
    =
    \begin{bmatrix}\text{strain}\\ \text{curvature}\end{bmatrix}.
\end{equation}

\begin{itemize}
\item \textbf{Sectional mass matrix}
\(\massmat(x)\in\R^{6\times 6}\):
\begin{equation}
    \massmat : [0,\ell] \to \R^{6\times 6}.
\end{equation}
The sectional mass matrix maps the linear and angular velocities of the
cross-section to the corresponding linear and angular momenta per unit length,
\begin{equation}
    \begin{bmatrix}\bm p\\ \bm h\end{bmatrix}
    =
    \massmat
    \begin{bmatrix}\bm v\\ \bm\omega\end{bmatrix}.
\end{equation}
It encodes the inertial properties of the beam and is assumed to be symmetric
positive definite. With the convention
\(\widetilde{\bm a}\bm b=\bm a\times\bm b\), it can be written in block form as
\begin{equation}
    \massmat
    =
    \begin{bmatrix}
        \bm M_1 & \bm M_2 \\
        \bm M_2^\top & \bm M_3
    \end{bmatrix}
    =
    \begin{bmatrix}
        \mu \idmat{3}{3} & -\mu \widetilde{\bm\zeta} \\
        \mu \widetilde{\bm\zeta} & \bm{\mathcal I}
    \end{bmatrix}.
    \label{eq:sectional_mass_matrix}
\end{equation}
Here \(\mu=\mu(x)>0\) is the mass per unit length,
\(\bm\zeta=\bm\zeta(x)\in\R^3\) is the position of the sectional mass center
relative to the chosen reference line, and
\(\bm{\mathcal I}=\bm{\mathcal I}(x)\in\R^{3\times3}\) is the sectional mass
moment of inertia tensor about the reference line. When the cross-section lies
in the \(x_2x_3\)-plane, this tensor is commonly written as
\begin{equation}
\bm{\mathcal I}
=
\begin{bmatrix}
    \mathcal I_{22}+\mathcal I_{33} & 0 & 0 \\
    0 & \mathcal I_{22} & \mathcal I_{23} \\
    0 & \mathcal I_{23} & \mathcal I_{33}
\end{bmatrix},
\end{equation}
where the sign convention for the product of inertia \(\mathcal I_{23}\) is
fixed by this definition. For a homogeneous prismatic beam whose reference line
passes through the sectional mass center and whose local \(x_2,x_3\) axes are
chosen as principal inertia axes, one has \(\bm\zeta=\zerovec{3}\) and
\(\mathcal I_{23}=0\). In this special case, \(\massmat\) is diagonal and
independent of \(x\).

\item \textbf{Sectional flexibility matrix}
\(\flexmat(x)\in\R^{6\times 6}\):
\begin{equation}
    \flexmat : [0,\ell] \to \R^{6\times 6}.
\end{equation}
The sectional flexibility, or compliance, matrix maps elastic
(rate-independent) sectional force and moment resultants to generalized strains
and curvatures. Denoting these elastic resultants by
\(\bm f_{\mathrm{el}}\) and \(\bm m_{\mathrm{el}}\), we have
\begin{equation}
    \begin{bmatrix}
        \bm\gamma \\
        \bm\kappa
    \end{bmatrix}
    =
    \flexmat
    \begin{bmatrix}
        \bm f_{\mathrm{el}} \\
        \bm m_{\mathrm{el}}
    \end{bmatrix}.
\end{equation}
In the notation used throughout this paper,
\([\bm f_{\mathrm{el}}^\top,\bm m_{\mathrm{el}}^\top]^\top=\utwo\). In the
Kelvin--Voigt model, the total sectional resultants additionally contain the
rate-dependent contribution \(\rtau\), so that
\([\bm f^\top,\bm m^\top]^\top=\utwo+\rtau\).
Equivalently, its inverse \(\bm K=\flexmat^{-1}\), when it exists, is the
sectional stiffness matrix. In this work, we assume that \(\flexmat(x)\) is
symmetric positive definite uniformly in \(x\). This excludes degenerate
idealizations such as perfectly inextensible or unshearable limits and is used
in the energy estimates.

The matrix is partitioned into \(3\times3\) blocks as
\begin{equation}
    \flexmat
    =
    \begin{bmatrix}
        \bm C_1 & \bm C_2 \\
        \bm C_2^\top & \bm C_3
    \end{bmatrix},
\end{equation}
where \(\bm C_1\) and \(\bm C_3\) are symmetric. The block \(\bm C_2\)
describes coupling between force-induced strains and moment-induced curvatures.

For a homogeneous prismatic isotropic beam, with the reference axis chosen
through the appropriate sectional center and the \(x_2,x_3\) axes aligned with
the principal bending axes, the extension, shear, torsion, and bending modes
decouple. In this special case, the flexibility matrix reduces to the diagonal
form
\begin{equation}
    \flexmat
    =
    \operatorname{diag}
    \left(
        \frac{1}{EA},
        \frac{1}{k_{s,2}GA},
        \frac{1}{k_{s,3}GA},
        \frac{1}{GJ},
        \frac{1}{EI_{22}},
        \frac{1}{EI_{33}}
    \right).
    \label{eq:sectional_flexibility_matrix}
\end{equation}
Here \(A\) is the cross-sectional area, \(E\) is Young's modulus, \(G\) is the
shear modulus, \(I_{22}\) and \(I_{33}\) are the principal second moments of
area, and \(k_{s,2}\) and \(k_{s,3}\) are shear correction factors. For circular
sections \(J=I_{22}+I_{33}\). More generally, \(J\) denotes the appropriate
Saint-Venant torsion constant; in engineering beam models it is sometimes
written as \(J=k_{s,1}(I_{22}+I_{33})\) with a torsional correction factor
\(k_{s,1}\).

\item \textbf{Kelvin--Voigt damping-coefficient matrix}
\(\dampmat(x)\in\R^{6\times 6}\).
The matrix governs the rate-dependent contribution to the sectional force and
moment resultants through the auxiliary variable
\(\rtau=\dampmat\dot{\utwo}\). No separate symmetry assumption is imposed on
\(\dampmat\). Instead, the effective Kelvin--Voigt operator
\(\Hmat=\dampmat\flexmat^{-1}\) is assumed uniformly symmetric positive
definite. This makes \(\dampmat=\Hmat\flexmat\) invertible and ensures that the
material dissipation density \(\rtau^\top\Hmat^{-1}\rtau\) is non-negative and
coercive.
\end{itemize}

\section{An Energy-Stable DGSEM in One Dimension}
\label{app:dgsem}
In this appendix, we briefly derive the Discontinuous Galerkin Spectral Element Method (DGSEM) discretization of the damped intrinsic beam equations that was implemented in \trixi\, to produce the numerical results in Section~\ref{sec:numerical_analysis}.
Moreover, we prove an energy-stability estimate for the resulting DGSEM discretization.
Our approach is based on \cite{Gassner2018}.
We start from the local semi-discrete formulation of the damped intrinsic beam equations obtained from the general DG discretization in Section~\ref{sec:discretization}.
For efficiency reasons, the local formulation will be transformed to a reference element.
The involved volume integrals on the reference element will be approximated by Legendre--Gauss--Lobatto quadratures.

The starting point to construct the DGSEM is the local semi-discrete formulation of the damped intrinsic beam equations on the element $\Omega_i$ (cf. Equation~\eqref{eq:local_semi_discrete_formulation}):
\begin{align}
\label{eq:appendix_local_semi_discrete_formulation_1}
    &\inproduct[\Omega_i]{\Gammamat \dot{\uh}}{\bm{v}_h}
    + \left[ \widehat{\Pimat\uh}^{\top}\bm{v}_h \right]_{x_{i-1}}^{x_{i}}
    - \inproduct[\Omega_i]{\Pimat \uh}{\bm{v}_h^\prime}
    + \left[ \widehat{\Pimat\bm{r}_h}^{\top}\bm{v}_h \right]_{x_{i-1}}^{x_{i}}
    - \inproduct[\Omega_i]{\Pimat \bm{r}_h}{\bm{v}_h^\prime} 
    = \inproduct[\Omega_i]{\bm Q(\uh,\bm{r}_h)}{\bm{v}_h},
    \\[1.25ex]
\label{eq:appendix_local_semi_discrete_formulation_2}
    &\inproduct[\Omega_i]{\Hmat^{-1}\rtauh}{\bm{w}_h}
    = \left[ \widehat{\uoneh}^{\top}\bm{w}_h \right]_{x_{i-1}}^{x_{i}}
    - \inproduct[\Omega_i]{\uoneh}{\bm{w}_h^\prime}
    - \inproduct[\Omega_i]{\Emat^\top \uoneh}{\bm{w}_h}
    + \inproduct[\Omega_i]{\Lone^\top(\uoneh) \flexmat \utwoh}{\bm{w}_h}.
\end{align}

All elements $\Omega_i$ of the partition $\mathcal{T}_h$ of the domain $\Omega$ are mapped onto a reference element $E:=[-1,1]$.
In particular, we use the affine transformation 
\begin{equation}
\label{eq:affine_mapping}
    x^{(i)}(\xi) = \frac{x_{i-1} + x_i}{2} + \frac{\Delta x_i}{2}\xi
\end{equation}
to map $\xi \in E$ to $x\in\Omega_i=[x_{i-1}, x_i]$.
Here, $\Delta x_i := x_i - x_{i-1}$ denotes the element length.
In practice, this has the advantage that the polynomial basis on $E$ and the quadrature weights can be reused across all elements.
For the solution restricted to a physical element $\Omega_i$, we introduce the notation
\begin{align}
    \uh^{(i)}(\xi,t)
    :=
    \left. \uh\left(x^{(i)}(\xi), t\right)\right|_{\xi\in E}.
\end{align}
The global numerical solution will continue to be referred to as $\uh$.

Under the change of variables, the local semi-discrete formulation~\eqref{eq:appendix_local_semi_discrete_formulation_1}--\eqref{eq:appendix_local_semi_discrete_formulation_2} for the element $\Omega_i$ reads
\begin{align}
\begin{split}
    \frac{\Delta x_i}{2}
    \inproduct[E]{\Gammamat \dot{\uh}^{(i)}}{\bm{v}_h^{(i)}}
    + \left[ \widehat{\Pimat\uh}^{\top}\bm{v}_h^{(i)} \right]_{x^{(i)}(-1)}^{x^{(i)}(1)}
    - \inproduct[E]{\Pimat \uh^{(i)}}{\bm{v}_h^{(i)\,\prime}}
    + \left[ \widehat{\Pimat\bm{r}_h}^{\top}\bm{v}_h^{(i)} \right]_{x^{(i)}(-1)}^{x^{(i)}(1)}
    - \inproduct[E]{\Pimat \bm{r}_h^{(i)}}{\bm{v}_h^{(i)\,\prime}} 
    \\
    = \frac{\Delta x_i}{2}
    \inproduct[E]{\bm Q\left(\uh^{(i)},\bm{r}_h^{(i)}\right)}{\bm{v}_h^{(i)}}&,
\end{split}
\\[2ex]
\begin{split}
    \frac{\Delta x_i}{2}
    \inproduct[E]{\Hmat^{-1}\rtauh^{(i)}}{\bm{w}_h^{(i)}}
    = 
    \left[ \widehat{\uoneh}^{\top}\bm{w}_h^{(i)} \right]_{x^{(i)}(-1)}^{x^{(i)}(1)}
    - \inproduct[E]{\uoneh^{(i)}}{\bm{w}_h^{(i)\,\prime}}
    - \frac{\Delta x_i}{2}
    \inproduct[E]{\Emat^\top \uoneh^{(i)}}{\bm{w}_h^{(i)}}
    \\
    + \frac{\Delta x_i}{2}
    \inproduct[E]{\Lone^\top\left(\uoneh^{(i)}\right) \flexmat \utwoh^{(i)}}{\bm{w}_h^{(i)}}&,
\end{split}
\end{align}
where we keep using $(\cdot)^\prime$ for the derivative with respect to the spatial variable, which is now $\xi$.

For the DGSEM, the solution and test functions are interpolated by polynomials of degree $k$ at the $k+1$ Legendre-Gauss-Lobatto (LGL) nodes, and all inner products are approximated by the corresponding LGL quadrature.
The associated discrete inner product is denoted by $\inproduct[E,k]{\cdot}{\cdot}$.
The following result connects this collocated realization to the exact-integral
DG analysis in Section~\ref{sec:discretization}.

\begin{corollary}[Energy stability of the LGL DGSEM]
\label{cor:dgsem_energy}
Under the coefficient, flux, and homogeneous-boundary hypotheses of
Theorem~\ref{thm:semidiscrete_energy}, the LGL-collocated DGSEM, using either
the alternating LDG or central BR1 auxiliary traces at interior interfaces and
\eqref{eq:homogeneous_aux_boundary_traces} at the physical boundaries,
satisfies the quadrature analogue of the energy--dissipation identity
\eqref{eq:disc_damp_eqn_final}. Here every element \(L^2\) inner product and
the associated energy and dual forcing norms are replaced by their LGL
quadrature counterparts.
It consequently satisfies the corresponding quadrature analogue of the
stability estimate~\eqref{eq:dg_final_bound}.
\end{corollary}

\begin{proof}
First, we apply the energy method choosing, elementwise, \(\bm v_h^{(i)}=\uh^{(i)}\) and
\(\bm w_h^{(i)}=\rtauh^{(i)}\).
The resulting local energy equation then reads
\begin{align}
\label{eq:dgsem_local_semi_discrete_energy_1}
\begin{split}
    \frac{\Delta x_i}{4}\frac{\mathrm{d}}{\mathrm{d}t} \dengnorm{\uh^{(i)}}^2
    +
    \left[ \widehat{\Pimat\uh}^{\top}\uh \right]_{x^{(i)}(-1)}^{x^{(i)}(1)}
    -
    \inproduct[E,k]{\Pimat \uh^{(i)}}{\uh^{(i)\,\prime}}
    +
    \left[ \widehat{\Pimat\bm{r}_h}^{\top}\uh \right]_{x^{(i)}(-1)}^{x^{(i)}(1)}
    -
    \inproduct[E,k]{\Pimat \bm{r}_h^{(i)}}{\uh^{(i)\,\prime}}
    \\=
    \frac{\Delta x_i}{2}\inproduct[E,k]{\bm Q\left(\uh^{(i)},\bm{r}_h^{(i)}\right)}{\uh^{(i)}},&
\end{split}
    \\[2ex]
\begin{split}
\label{eq:dgsem_local_semi_discrete_energy_2}
    \frac{\Delta x_i}{2}
    \inproduct[E,k]{\Hmat^{-1}\rtauh^{(i)}}{\rtauh^{(i)}}
    =
    \left[ \widehat{\uoneh}^{\top}\rtauh \right]_{x^{(i)}(-1)}^{x^{(i)}(1)}
    -
    \inproduct[E,k]{\uoneh^{(i)}}{\rtauh^{(i)\,\prime}}
    -
    \frac{\Delta x_i}{2} \inproduct[E,k]{\Emat^\top \uoneh^{(i)}}{\rtauh^{(i)}}
    \\+
    \frac{\Delta x_i}{2}\inproduct[E,k]{\Lone^\top\left(\uoneh^{(i)}\right) \flexmat \utwoh^{(i)}}{\rtauh^{(i)}}.&
\end{split}
\end{align}
Here,
\begin{equation}
\dengnorm{\bm z}^{2}
:=
\inproduct[E,k]{\Gammamat \bm z}{\bm z},
\end{equation}
and we use that \(\Gammamat\) is symmetric and time-independent.

When passing to the DGSEM, all quantities entering the discrete inner products are evaluated at the LGL nodes and represented by their nodal interpolants in \([P_k(E)]^d\), with \(d=6\) or \(d=12\). Since interpolation and quadrature nodes coincide, we suppress the interpolation operator in the notation. The cancellations used below are understood as nodal algebraic cancellations in the LGL quadrature inner product.

For the volume terms 
\begin{align}
    \inproduct[E,k]{\Pimat \uh^{(i)}}{\uh^{(i)\,\prime}}
\end{align}
in the first equation and
\begin{align}
    \inproduct[E,k]{\uoneh^{(i)}}{\rtauh^{(i)\,\prime}}
\end{align}
in the second equation, we exploit the \emph{summation-by-parts} property of the Legendre--Gauss--Lobatto quadrature \cite{Kopriva2010,gassner2013skew}.
In particular, we have
\begin{align}
    \inproduct[E,k]{\Pimat \uh^{(i)}}{\uh^{(i)\,\prime}}
    &=
    \frac{1}{2}\left[ \uh^\top \Pimat\uh \right]_{x^{(i)}(-1)}^{x^{(i)}(1)},
    \\[2ex]
     \inproduct[E,k]{\uoneh^{(i)}}{\rtauh^{(i)\,\prime}}
     &= 
     \left[ \uoneh^\top\rtauh \right]_{x^{(i)}(-1)}^{x^{(i)}(1)} - \inproduct[E,k]{\uoneh^{(i)\,\prime}}{\rtauh^{(i)}}.
\end{align}
Substituting this back into equations~\eqref{eq:dgsem_local_semi_discrete_energy_1}--\eqref{eq:dgsem_local_semi_discrete_energy_2} yields
\begin{align}
\begin{split}
    \frac{\Delta x_i}{4}\frac{\mathrm{d}}{\mathrm{d}t} \dengnorm{\uh^{(i)}}^2
    +
    \left[ \left(\widehat{\Pimat\uh} - \frac{1}{2}\Pimat\uh\right)^\top\uh \right]_{x^{(i)}(-1)}^{x^{(i)}(1)}
    +
    \left[ \widehat{\Pimat\bm{r}_h}^{\top}\uh \right]_{x^{(i)}(-1)}^{x^{(i)}(1)}
    -
    \inproduct[E,k]{\Pimat \bm{r}_h^{(i)}}{\uh^{(i)\,\prime}}
    \\=
    \frac{\Delta x_i}{2}\inproduct[E,k]{\bm Q\left(\uh^{(i)},\bm{r}_h^{(i)}\right)}{\uh^{(i)}},&
\end{split}
    \\[2ex]
\begin{split}
    \frac{\Delta x_i}{2}
    \inproduct[E,k]{\Hmat^{-1}\rtauh^{(i)}}{\rtauh^{(i)}}
    =
    \left[ \left(\widehat{\uoneh} - \uoneh\right)^\top\rtauh \right]_{x^{(i)}(-1)}^{x^{(i)}(1)}
    +
    \inproduct[E,k]{\uoneh^{(i)\,\prime}}{\rtauh^{(i)}}
    -
    \frac{\Delta x_i}{2} \inproduct[E,k]{\Emat^\top \uoneh^{(i)}}{\rtauh^{(i)}}
    \\+
    \frac{\Delta x_i}{2}\inproduct[E,k]{\Lone^\top\left(\uoneh^{(i)}\right) \flexmat \utwoh^{(i)}}{\rtauh^{(i)}}.&
\end{split}
\end{align}
Adding up the two equations and rearranging terms leads to the following local energy equation:
\begin{align}
\label{eq:dgsem_local_energy_conservation}
\begin{split}
    \frac{\Delta x_i}{4}\frac{\mathrm{d}}{\mathrm{d}t} \dengnorm{\uh^{(i)}}^2
    +
    &\frac{\Delta x_i}{2}\inproduct[E,k]{\Hmat^{-1}\rtauh^{(i)}}{\rtauh^{(i)}}
    \\=
    &-
    \left[ \left(\widehat{\Pimat\uh} - \frac{1}{2}\Pimat\uh\right)^\top\uh \right]_{x^{(i)}(-1)}^{x^{(i)}(1)}
    -
    \left[ \widehat{\Pimat\bm{r}_h}^{\top}\uh \right]_{x^{(i)}(-1)}^{x^{(i)}(1)}
    \\
    &+
    \left[ \left(\widehat{\uoneh} - \uoneh\right)^{\top}\rtauh \right]_{x^{(i)}(-1)}^{x^{(i)}(1)}
    +
    \frac{\Delta x_i}{2}\inproduct[E,k]{\qext^{(i)}}{\uh^{(i)}}.
    \end{split}
\end{align}
As in the continuous case (cf. Equation~\eqref{eq:source_term_simplified}), the source term together with the terms
\begin{align}
    -\frac{\Delta x_i}{2}
    \inproduct[E,k]{\Emat^\top \uoneh^{(i)}}{\rtauh^{(i)}},
    \hspace{10ex}
    +\frac{\Delta x_i}{2}
    \inproduct[E,k]{\Lone^\top\left(\uoneh^{(i)}\right) \flexmat \utwoh^{(i)}}{\rtauh^{(i)}}
\end{align}
reduce to the external forcing $\frac{\Delta x_i}{2}\inproduct[E,k]{\qext}{\uh^{(i)}}$. The cancellation relies on the symmetry properties~\eqref{eq:l_operator_properties_1}--\eqref{eq:l_operator_properties_2} of $\Lone$ and $\Ltwo$, which hold pointwise and therefore remain valid in the LGL collocation inner product.
Moreover, the terms 
\begin{align}
    -\inproduct[E,k]{\Pimat \bm{r}_h^{(i)}}{\uh^{(i)\,\prime}},
    \hspace{10ex}
    \inproduct[E,k]{\uoneh^{(i)\,\prime}}{\rtauh^{(i)}}
\end{align}
cancel out each other due to the structure of $\Pimat$ and the definition of $\bm{r}_h$.

Equation~\eqref{eq:dgsem_local_energy_conservation} proves an expected element-wise energy balance: within one element, apart from damping, energy changes only due to the presence of external forces and moments or through the boundaries of the element.

On assembly, the auxiliary-interface contribution vanishes by
\eqref{eq:aux_trace_compatibility}; hence the DGSEM contraction covers both
alternating LDG and central BR1 at interior interfaces. Both use the same
data-carrying physical-boundary traces from
Paragraph~\ref{par:boundary_treatment}, rather than central boundary averages.

We continue by summing the element-wise energy contribution over all elements to obtain the global energy of the solution.
The global semi-discrete energy equation reads
\begin{align}
\label{eq:dgsem_total_energy_raw}
\begin{split}
    &\frac{1}{2}\frac{\mathrm{d}}{\mathrm{d}t} \sum_{i=1}^{N_c}
    \frac{\Delta x_i}{2} \dengnorm{\uh^{(i)}}^2
    \\=
    &-
    \sum_{i=1}^{N_c}
    \frac{\Delta x_i}{2} \inproduct[E,k]{\Hmat^{-1}\rtauh^{(i)}}{\rtauh^{(i)}}
    -
    \left[
    \left( \widehat{\Pimat\uh}
    -
    \frac{1}{2}\Pimat\uh\right)^\top\uh
    \right]_{x^{(1)}(-1)}^{x^{(N_c)}(1)}
    -
    \sum_{i\in\mathcal{I}_{\mathrm{int}}}
    \interfacepoint[x_i]
    {\frac{\sigma}{2}\Gammamat|\Amat|\jump{\uh}}
    {\jump{\uh}}
    \\
    &+
    \left[
    \widehat{\uoneh}^\top\rtauh
    \right]_{x^{(1)}(-1)}^{x^{(N_c)}(1)}
    -
    \left[
    \widehat{\Pimat\bm{r}_h}^\top \uh
    \right]_{x^{(1)}(-1)}^{x^{(N_c)}(1)}
    -
    \left[
    \uoneh^\top \rtauh
    \right]_{x^{(1)}(-1)}^{x^{(N_c)}(1)}
    +
    \sum_{i=1}^{N_c}
    \frac{\Delta x_i}{2}
    \inproduct[E,k]{\uh^{(i)}}{\qext^{(i)}}.
\end{split}
\end{align}
where we proceeded analogously to the energy analysis in the semi-discrete context (cf. Section~\ref{subsec:discrete_energy_stability}).

The matrix $\Hmat$ is positive definite, so that
\begin{equation}
    -\inproduct[E,k]{\Hmat^{-1}\rtauh^{(i)}}{\rtauh^{(i)}} \leq 0,
    \hspace{4ex}
    i\in\{1,\ldots, N_c\}.
\end{equation}
Furthermore, in Section~\ref{sec:discretization} it was shown that for the numerical flux at the physical boundaries, which we chose in this paper, we have
\begin{align}
    \left[
    \left( \widehat{\Pimat\uh}
    -
    \frac{1}{2}\Pimat\uh\right)^\top\uh
    \right]_{x^{(1)}(-1)}^{x^{(N_c)}(1)}
    \geq
    0
\end{align}
(cf. Equations~\eqref{eq:advection_flux_left}--~\eqref{eq:advection_flux_right}) and
\begin{align}
    \left[
    \left( \widehat{\uoneh} \right)^\top\rtauh
    \right]_{x^{(1)}(-1)}^{x^{(N_c)}(1)}
    -
    \left[
    \widehat{\Pimat\bm{r}_h}^{\top} \uh
    \right]_{x^{(1)}(-1)}^{x^{(N_c)}(1)}
    -
    \left[
    \uoneh^\top \rtauh
    \right]_{x^{(1)}(-1)}^{x^{(N_c)}(1)}
    =
    0
\end{align}
(cf. Equation~\eqref{eq:diffusion_flux_bnd_left} and the analogous identity at $x=\ell$).
Combining these estimates yields
\begin{align}
    \frac{1}{2}\frac{\mathrm{d}}{\mathrm{d}t} \sum_{i=1}^{N_c} 
    \frac{\Delta x_i}{2} \dengnorm{\uh^{(i)}}^2
    \leq
    \sum_{i=1}^{N_c}
    \frac{\Delta x_i}{2}\inproduct[E,k]{\uh^{(i)}}{\qext^{(i)}}
\end{align}
which is the differential energy inequality for bounded external sources
(cf.~\cite{Bleffert2025}). Applying the same weighted Cauchy--Schwarz, Young,
and Grönwall argument as in
\eqref{eq:pointwise_bound}--\eqref{eq:dg_final_bound}, now in the LGL quadrature
inner product, gives the asserted quadrature stability estimate and completes
the proof.
\end{proof}

\section{Reformulation as a Nonlinear Advection--Diffusion System}
\label{app:nonlinear_advection_diffusion}

For completeness, we provide the explicit matrix operators appearing in the reformulated system~\eqref{eq:advection_diffusion_eqn}. Throughout this appendix the coefficient matrices $\Emat$, $\flexmat$, $\Hmat=\dampmat\flexmat^{-1}$ are assumed independent of $x$; spatially varying coefficients generate additional lower-order terms involving their derivatives, which we omit here. Substituting $\dampmat\dot{\utwo}=\Hmat\big(\uone'-\Emat^\top\uone+\Lone^\top(\uone)\flexmat\utwo\big)$ (which follows from~\eqref{eq:damped_equation_2}) into~\eqref{eq:damped_equation_1}, expanding the spatial derivative by the product rule, and collecting terms with the help of the identities $\Lone^\top(\bm a)\bm b=-\Lone^\top(\bm b)\bm a$ and $\Ltwo(\bm a)\bm b=-\Lone(\bm b)\bm a$ rewrites the damped intrinsic beam equations in the form
\begin{equation}
    \Gammamat\,\dot{\bm{u}}
    =
    \mathcal{D}\,\bm{u}^{\prime\prime}
    + \mathcal{A}(\bm{u})\,\bm{u}^\prime
    + \mathcal{S}(\bm{u})\,\bm{u}
    +
    \begin{bmatrix}
        \bm{f}_e\\[0.2em]
        \zerovec{6}
    \end{bmatrix}.
\end{equation}
The operators $\mathcal{D}$, $\mathcal{A}(\bm{u})$, and $\mathcal{S}(\bm{u})$ are defined by
\begin{align}
\mathcal{D} &=
\begin{bmatrix}
    \Hmat & \zeromat{6}{6}\\
    \zeromat{6}{6} & \zeromat{6}{6}
\end{bmatrix}, \\[0.8em]
\mathcal{A}(\bm{u}) &=
\begin{bmatrix}
    \Emat\Hmat + \Lone(\flexmat \utwo)\Hmat
    - \Hmat\Emat^\top - \Hmat\Lone^\top(\flexmat\utwo)
    &
    \idmat{6}{6} + \Hmat \Lone^\top(\uone)\flexmat \\[0.4em]
    \idmat{6}{6}
    &
    \zeromat{6}{6}
\end{bmatrix}, \\[0.8em]
\mathcal{S}(\bm{u}) &=
\begin{bmatrix}
    -\Emat\Hmat\Emat^\top - \Lone(\uone)\massmat
    - \Lone(\flexmat\utwo)\Hmat\Emat^\top &
    \Emat + \Emat\Hmat \Lone^\top(\uone)\flexmat
    - \Ltwo(\utwo)\flexmat
    + \Lone(\flexmat\utwo)\Hmat\Lone^\top(\uone)\flexmat \\[0.4em]
    -\Emat^\top &
    \Lone^\top(\uone)\flexmat
\end{bmatrix}\!.
\end{align}

Here, $\mathcal{D}$ collects the diffusive contributions induced by Kelvin--Voigt damping, $\mathcal{A}(\bm{u})$ contains the nonlinear transport-like terms associated with spatial gradients, and $\mathcal{S}(\bm{u})$ represents nonlinear source contributions. Although no physical advection is present, this structure reveals that the governing equations possess the analytical form of a nonlinear hyperbolic--parabolic system, closely resembling advection-diffusion models from a mathematical perspective.

\section{Computations of Flux Matrices}
\label{app:flux_matrices}

In this appendix, we collect the algebraic derivations underlying the flux splitting,
the weighted flux matrices, and the boundary stability terms used in the discrete
energy analysis.

\subsection{Eigendecomposition}
As defined in Section~\ref{subsec:numerical_flux}, we consider the matrix
\begin{equation}
    \Psimat = \flexmat^{-1/2}\massmat^{-1}\flexmat^{-1/2}.
\end{equation}
Since both \(\flexmat\) and \(\massmat\) are symmetric positive definite, it follows that
\(\Psimat\) is symmetric positive definite as well. Indeed, for any nonzero
\(\bm{z} \in \R^{6}\),
\begin{equation}
    \bm{z}^\top \Psimat \bm{z}
    = \bm{z}^\top \flexmat^{-1/2} \massmat^{-1} \flexmat^{-1/2} \bm{z}
    = (\flexmat^{-1/2} \bm{z})^\top \massmat^{-1} (\flexmat^{-1/2} \bm{z})
    > 0.
\end{equation}

Hence all eigenvalues of \(\Psimat\) are strictly positive, and there exists an
orthogonal matrix \(\bm{\mathcal{X}} \in \R^{6\times 6}\) and a diagonal
matrix \(\Lambda_\psi \in \R^{6\times 6}\) with positive entries
such that the spectral decomposition
\begin{equation}
    \Psimat = \bm{\mathcal{X}}^\top \Lambda_\psi^2 \bm{\mathcal{X}}
    \label{eq:psi}
\end{equation}
holds.

Based on this diagonalization, we construct the transformation matrices used to
diagonalize the matrix \(\Amat\) (see Section~\ref{subsec:bc}):
\begin{equation}
    \bm{T} \coloneqq
    \begin{bmatrix}
        \flexmat^{1/2} \bm{\mathcal{X}}^\top & \flexmat^{1/2} \bm{\mathcal{X}}^\top \\
        \flexmat^{-1/2} \bm{\mathcal{X}}^\top \Lambda_\psi^{-1} &
        -\flexmat^{-1/2} \bm{\mathcal{X}}^\top \Lambda_\psi^{-1}
    \end{bmatrix},
    \qquad
    \bm{T}^{-1} \coloneqq \frac{1}{2}
    \begin{bmatrix}
        \bm{\mathcal{X}} \flexmat^{-1/2} &
        \Lambda_\psi \bm{\mathcal{X}} \flexmat^{1/2} \\
        \bm{\mathcal{X}} \flexmat^{-1/2} &
        -\Lambda_\psi \bm{\mathcal{X}} \flexmat^{1/2}
    \end{bmatrix}.
\end{equation}

\subsection{Characteristic Flux Splitting}
The matrices \(\Amat^+\) and \(\Amat^-\) arise from the spectral decomposition of the matrix \(\Amat\) and represent its decomposition into contributions associated with positive and negative characteristic speeds, respectively. For this reason, we refer to them as the \emph{positive and negative characteristic flux matrices}. They are defined through the similarity transformations
\begin{equation}
\Amat^+ = \bm{T}\,\mathbf{\Lambda}^+\,\bm{T}^{-1},
\qquad
\Amat^- = \bm{T}\,\mathbf{\Lambda}^-\,\bm{T}^{-1},
\end{equation}
where \(\mathbf{\Lambda}=\operatorname{diag}(-\Lambda_\psi,\Lambda_\psi)\) collects the eigenvalues of \(\Amat\), and \(\mathbf{\Lambda}^+=\operatorname{diag}(\zeromat{6}{6},\Lambda_\psi)\), \(\mathbf{\Lambda}^-=\operatorname{diag}(-\Lambda_\psi,\zeromat{6}{6})\) contain its positive and negative parts, respectively.

Expanding the expression for \(\Amat^+\) yields
\begin{align}
\Amat^+
&=
\begin{bmatrix}
    \flexmat^{\frac{1}{2}} \bm{\mathcal{X}}^\top &
    \flexmat^{\frac{1}{2}} \bm{\mathcal{X}}^\top \\[0.2em]
    \flexmat^{-\frac{1}{2}} \bm{\mathcal{X}}^\top \Lambda_\psi^{-1} &
    -\flexmat^{-\frac{1}{2}} \bm{\mathcal{X}}^\top \Lambda_\psi^{-1}
\end{bmatrix}
\begin{bmatrix}
    \zeromat{6}{6} & \zeromat{6}{6} \\
    \zeromat{6}{6} & \Lambda_\psi
\end{bmatrix}
\bm{T}^{-1} \\[1em]
&=
\frac{1}{2}
\begin{bmatrix}
\flexmat^{\frac{1}{2}}\bm{\mathcal{X}}^\top \Lambda_\psi\bm{\mathcal{X}}\flexmat^{-\frac{1}{2}}
&
-\flexmat^{\frac{1}{2}}\bm{\mathcal{X}}^\top \Lambda_\psi^2\bm{\mathcal{X}}\flexmat^{\frac{1}{2}} \\[0.5em]
-\flexmat^{-\frac{1}{2}}\bm{\mathcal{X}}^\top\bm{\mathcal{X}}\flexmat^{-\frac{1}{2}}
&
\flexmat^{-\frac{1}{2}}\bm{\mathcal{X}}^\top \Lambda_\psi\bm{\mathcal{X}}\flexmat^{\frac{1}{2}}
\end{bmatrix}.
\end{align}

Using \(\bm{\mathcal{X}}^\top \bm{\mathcal{X}} = \mathbf{I}\) and the identity \eqref{eq:psi}, this simplifies to
\begin{equation}
\Amat^+ =
\frac{1}{2}
\begin{bmatrix}
\flexmat^{\frac{1}{2}}\Psimat^{\frac{1}{2}}\flexmat^{-\frac{1}{2}} & -\massmat^{-1} \\
-\flexmat^{-1} & \flexmat^{-\frac{1}{2}}\Psimat^{\frac{1}{2}}\flexmat^{\frac{1}{2}}
\end{bmatrix}.
\end{equation}

Similarly, the negative characteristic flux matrix is given by
\begin{equation}
\Amat^- =
\frac{1}{2}
\begin{bmatrix}
-\flexmat^{\frac{1}{2}}\Psimat^{\frac{1}{2}}\flexmat^{-\frac{1}{2}} & -\massmat^{-1} \\
-\flexmat^{-1} & -\flexmat^{-\frac{1}{2}}\Psimat^{\frac{1}{2}}\flexmat^{\frac{1}{2}}
\end{bmatrix}.
\end{equation}

By construction, these matrices satisfy the consistency relation
\begin{equation}
\Amat = \Amat^+ + \Amat^-,
\end{equation}
which reflects the exact decomposition of the matrix \(\Amat\) into right- and left-propagating characteristic contributions.

\medskip
We now examine the weighted characteristic flux matrices.
For \(\Amat^+\), we obtain
\begin{align}
\Gammamat\Amat^+ &= \begin{bmatrix}
    \massmat & \zeromat{6}{6} \\
    \zeromat{6}{6} & \flexmat
\end{bmatrix} \cdot \frac{1}{2} \begin{bmatrix}
    \flexmat^\frac{1}{2} \Psimat^\frac{1}{2} \flexmat^{-\frac{1}{2}} & -\massmat^{-1} \\
    -\flexmat^{-1} & \flexmat^{-\frac{1}{2}} \Psimat^\frac{1}{2} \flexmat^\frac{1}{2}
\end{bmatrix} = \frac{1}{2} \begin{bmatrix}
    \flexmat^{-\frac{1}{2}} \Psimat^{-\frac{1}{2}} \flexmat^{-\frac{1}{2}} & -\idmat{6}{6} \\
    -\idmat{6}{6} & \flexmat^\frac{1}{2} \Psimat^\frac{1}{2} \flexmat^\frac{1}{2}
\end{bmatrix},
\end{align}
while for \(\Amat^-\) we analogously obtain
\begin{align}
\Gammamat\Amat^- &= \begin{bmatrix}
    \massmat & \zeromat{6}{6} \\
    \zeromat{6}{6} & \flexmat
\end{bmatrix} \cdot \frac{1}{2} \begin{bmatrix}
    -\flexmat^\frac{1}{2} \Psimat^\frac{1}{2} \flexmat^{-\frac{1}{2}} & -\massmat^{-1} \\
    -\flexmat^{-1} & -\flexmat^{-\frac{1}{2}} \Psimat^\frac{1}{2} \flexmat^\frac{1}{2}
\end{bmatrix}= \frac{1}{2}\begin{bmatrix}
    -\flexmat^{-\frac{1}{2}} \Psimat^{-\frac{1}{2}} \flexmat^{-\frac{1}{2}} & -\idmat{6}{6} \\
    -\idmat{6}{6} & -\flexmat^\frac{1}{2} \Psimat^\frac{1}{2} \flexmat^\frac{1}{2}
\end{bmatrix}.
\end{align}

\subsection{Absolute Flux Matrix}

The absolute value of the matrix \(\Amat\) is defined via its spectral decomposition as
\begin{equation}
|\Amat| := \bm{T}\,|\mathbf{\Lambda}|\,\bm{T}^{-1},
\end{equation}
where \(|\mathbf{\Lambda}|\) is the diagonal matrix obtained by replacing each eigenvalue of \(\mathbf{\Lambda}\) with its absolute value.

The positive and negative spectral parts satisfy
\begin{equation}
\Amat^+=\tfrac12(\Amat+|\Amat|),
\qquad
\Amat^-=\tfrac12(\Amat-|\Amat|).
\end{equation}
Consequently, \( |\Amat|=\Amat^+-\Amat^- \), and hence
$\Gammamat|\Amat| = \Gammamat\Amat^+ - \Gammamat\Amat^-$.
Substituting the expressions derived in the previous subsection, we obtain
\begin{equation}
\Gammamat|\Amat| =
\begin{bmatrix}
\flexmat^{-\frac{1}{2}}\Psimat^{-\frac{1}{2}}\flexmat^{-\frac{1}{2}} & \zeromat{6}{6} \\
\zeromat{6}{6} & \flexmat^{\frac{1}{2}}\Psimat^{\frac{1}{2}}\flexmat^{\frac{1}{2}}
\end{bmatrix}.
\end{equation}

Both diagonal blocks in this expression are symmetric positive definite.
Consequently, the weighted absolute-flux matrix
\(\Gammamat|\Amat|\) is symmetric positive definite. This property is central
to the construction of energy-stable numerical fluxes and the ensuing
discrete stability analysis.

\section{Boundary Contribution Analysis}
\label{app:boundary_contribution_analysis}

In this section, we analyze the stability of the boundary treatment by
evaluating the energy contribution
\((\widehat{\Pimat\uh}-\tfrac12\Pimat\uh)^\top\uh\) at the domain boundaries.
Consistent with Section~\ref{subsec:discrete_energy_stability}, we use the
fully upwind characteristic flux \(\sigma=1\), for which
\(\widehat{\Pimat\uh}
=\Gammamat\Amat^+\uh^-+\Gammamat\Amat^-\uh^+\), together with the homogeneous
ghost states~\eqref{eq:homogeneous_outer_states}. For inhomogeneous data, the
computation produces the additional terms recorded in
Remark~\ref{rem:inhomogeneous_bc_data}.

\subsection{\texorpdfstring{Left Boundary ($x=0$)}{Left Boundary (x=0)}}
At the left boundary, the stability term expands as:
\begingroup
\allowdisplaybreaks
\begin{align}
\left( \widehat{\Pimat\uh} - \tfrac{1}{2}\Pimat \uh \right)^{\!\top}\uh \Big|_{0}
&=
\Bigl(\bigl(\Gammamat\Amat^{+}\bigr)\uh^{-}
      +\bigl(\Gammamat\Amat^{-}\bigr)\uh^{+}
      -\tfrac{1}{2}\Pimat\uh^{+}\Bigr)^{\!\top}\uh^{+}.
\label{eq:left_boundary_start}
\end{align}
We now evaluate each term in the sum individually. 
\begin{align}
\bigl((\Gammamat\Amat^{+})\uh^{-}\bigr)^{\!\top}\uh^{+}
&=
\left(
\frac{1}{2}
\begin{bmatrix}
\flexmat^{-\frac{1}{2}}\Psimat^{-\frac{1}{2}}\flexmat^{-\frac{1}{2}}
    & -\idmat{6}{6}\\[4pt]
-\idmat{6}{6}
    & \flexmat^{\frac{1}{2}}\Psimat^{\frac{1}{2}}\flexmat^{\frac{1}{2}}
\end{bmatrix}
\begin{bmatrix}
\zerovec{6}\\[4pt]
\utwoh^{+}
+
\flexmat^{-\frac{1}{2}}\Psimat^{-\frac{1}{2}}\flexmat^{-\frac{1}{2}}\uoneh^{+}
\end{bmatrix}
\right)^{\!\top}
\begin{bmatrix}
\uoneh^{+}\\[6pt]
\utwoh^{+}
\end{bmatrix}
\nonumber \\[6pt]
&=
-\tfrac{1}{2}\bigl(\utwoh^{+} + \flexmat^{-\frac{1}{2}}\Psimat^{-\frac{1}{2}}\flexmat^{-\frac{1}{2}}\uoneh^{+}\bigr)^{\!\top}\uoneh^{+} +
\tfrac{1}{2}\bigl(\flexmat^{\frac{1}{2}}\Psimat^{\frac{1}{2}}\flexmat^{\frac{1}{2}}
    \bigl(\utwoh^{+} + \flexmat^{-\frac{1}{2}}\Psimat^{-\frac{1}{2}}\flexmat^{-\frac{1}{2}}\uoneh^{+}\bigr)\bigr)^{\!\top}\utwoh^{+}.
\label{eq:g_plus_term}
\end{align}

Next, the negative characteristic contribution is
\begin{align}
\bigl((\Gammamat\Amat^{-})\uh^{+}\bigr)^{\!\top}\uh^{+}
&=
\left(
\frac{1}{2}
\begin{bmatrix}
-\,\flexmat^{-\frac{1}{2}}\Psimat^{-\frac{1}{2}}\flexmat^{-\frac{1}{2}}
    & -\idmat{6}{6}\\[4pt]
- \idmat{6}{6}
    & -\,\flexmat^{\frac{1}{2}}\Psimat^{\frac{1}{2}}\flexmat^{\frac{1}{2}}
\end{bmatrix}
\begin{bmatrix}
\uoneh^{+}\\[6pt]
\utwoh^{+}
\end{bmatrix}
\right)^{\!\top}
\begin{bmatrix}
\uoneh^{+}\\[6pt]
\utwoh^{+}
\end{bmatrix}
\nonumber \\[6pt]
&=
-\tfrac{1}{2}\bigl(\flexmat^{-\frac{1}{2}}\Psimat^{-\frac{1}{2}}\flexmat^{-\frac{1}{2}}\uoneh^{+} + \utwoh^{+}\bigr)^{\!\top}\uoneh^{+} \nonumber \\
&\quad \;-\;
\tfrac{1}{2}\bigl(\uoneh^{+} + \flexmat^{\frac{1}{2}}\Psimat^{\frac{1}{2}}\flexmat^{\frac{1}{2}}\utwoh^{+}\bigr)^{\!\top}\utwoh^{+}.
\label{eq:g_minus_term}
\end{align}
The subtracted physical-flux contribution is
\begin{equation}
\Bigl(\tfrac{1}{2}\Pimat\uh^{+}\Bigr)^{\!\top}\uh^{+}
= -\,\uoneh^{+}{}^{\!\top}\utwoh^{+}.
\label{eq:pimat_term}
\end{equation}
Substituting \eqref{eq:g_plus_term}--\eqref{eq:pimat_term} into
\eqref{eq:left_boundary_start} gives the left-boundary contraction
\begin{align}
\MoveEqLeft
\Bigl(\bigl(\Gammamat\Amat^{+}\bigr)\uh^{-}
      +\bigl(\Gammamat\Amat^{-}\bigr)\uh^{+}
      -\tfrac{1}{2}\Pimat\uh^{+}\Bigr)^{\!\top}\uh^{+}
\nonumber \\[6pt]
&=
- \bigl(\uoneh^{+}\bigr)^{\!\top}
    \bigl(\flexmat^{-\frac{1}{2}}\Psimat^{-\frac{1}{2}}\flexmat^{-\frac{1}{2}}\bigr)\uoneh^{+}.
\label{eq:left_boundary_final}
\end{align} 
\endgroup

\subsection{\texorpdfstring{Right Boundary ($x=\ell$)}{Right Boundary (x=l)}}
Similarly, at the right boundary, we evaluate:
\begingroup
\allowdisplaybreaks
\begin{align}
\left( \widehat{\Pimat\uh} - \tfrac{1}{2}\Pimat \uh \right)^{\!\top}\uh \Big|_{\ell}
&=
\Bigl(\bigl(\Gammamat\Amat^{+}\bigr)\uh^{-}
      +\bigl(\Gammamat\Amat^{-}\bigr)\uh^{+}
      -\tfrac{1}{2}\Pimat\uh^{-}\Bigr)^{\!\top}\uh^{-}.
\label{eq:right_boundary_start}
\end{align}
The positive characteristic contribution is
\begin{align}
\bigl((\Gammamat\Amat^{+})\uh^{-}\bigr)^{\!\top}\uh^{-}
&=
\left(
\frac{1}{2}
\begin{bmatrix}
\flexmat^{-\frac{1}{2}}\Psimat^{-\frac{1}{2}}\flexmat^{-\frac{1}{2}}
    & -\idmat{6}{6}\\[4pt]
-\idmat{6}{6}
    & \flexmat^{\frac{1}{2}}\Psimat^{\frac{1}{2}}\flexmat^{\frac{1}{2}}
\end{bmatrix}
\begin{bmatrix}
\uoneh^{-}\\[6pt]
\utwoh^{-}
\end{bmatrix}
\right)^{\!\top}
\begin{bmatrix}
\uoneh^{-}\\[6pt]
\utwoh^{-}
\end{bmatrix}
\nonumber \\[6pt]
&=
\tfrac{1}{2}\bigl(\flexmat^{-\frac{1}{2}}\Psimat^{-\frac{1}{2}}\flexmat^{-\frac{1}{2}}\uoneh^{-}-\utwoh^{-}\bigr)^{\!\top}\uoneh^{-} \nonumber \\
&\quad \;+\;
\tfrac{1}{2}\bigl(-\uoneh^{-}+\flexmat^{\frac{1}{2}}\Psimat^{\frac{1}{2}}\flexmat^{\frac{1}{2}}\utwoh^{-}\bigr)^{\!\top}\utwoh^{-}.
\label{eq:g_plus_right}
\end{align}
After substituting the homogeneous right ghost state, the negative
characteristic contribution is
\begin{align}
\bigl((\Gammamat\Amat^{-})\uh^{+}\bigr)^{\!\top}\uh^{-}
&=
\left(
\frac{1}{2}
\begin{bmatrix}
-\,\flexmat^{-\frac{1}{2}}\Psimat^{-\frac{1}{2}}\flexmat^{-\frac{1}{2}}
    & -\idmat{6}{6}\\[4pt]
- \idmat{6}{6}
    & -\,\flexmat^{\frac{1}{2}}\Psimat^{\frac{1}{2}}\flexmat^{\frac{1}{2}}
\end{bmatrix}
\begin{bmatrix}
\uoneh^{-}-\flexmat^{\frac{1}{2}}\Psimat^{\frac{1}{2}}\flexmat^{\frac{1}{2}}\utwoh^{-}\\[6pt]
\zerovec{6}
\end{bmatrix}
\right)^{\!\top}
\begin{bmatrix}
\uoneh^{-}\\[6pt]
\utwoh^{-}
\end{bmatrix}
\nonumber \\[6pt]
&=
-\tfrac{1}{2}\bigl(\flexmat^{-\frac{1}{2}}\Psimat^{-\frac{1}{2}}\flexmat^{-\frac{1}{2}}
    (\uoneh^{-}-\flexmat^{\frac{1}{2}}\Psimat^{\frac{1}{2}}\flexmat^{\frac{1}{2}}\utwoh^{-})\bigr)^{\!\top}\uoneh^{-} \nonumber \\
&\quad \;-\;
\tfrac{1}{2}\bigl(\uoneh^{-}-\flexmat^{\frac{1}{2}}\Psimat^{\frac{1}{2}}\flexmat^{\frac{1}{2}}\utwoh^{-}\bigr)^{\!\top}\utwoh^{-}.
\label{eq:g_minus_right}
\end{align}
The subtracted physical-flux contribution at the right boundary is
\begin{equation}
\Bigl(\tfrac{1}{2}\Pimat\uh^{-}\Bigr)^{\!\top}\uh^{-}
= -\,\uoneh^{-}{}^{\!\top}\utwoh^{-}.
\label{eq:pimat_right}
\end{equation}
Combining these expressions, we obtain the final quadratic form for the right boundary contribution:
\begin{align}
\MoveEqLeft
\Bigl(\bigl(\Gammamat\Amat^{+}\bigr)\uh^{-}
      +\bigl(\Gammamat\Amat^{-}\bigr)\uh^{+}
      -\tfrac{1}{2}\Pimat\uh^{-}\Bigr)^{\!\top}\uh^{-}
\nonumber \\[6pt]
&=
(\utwoh^{-})^{\!\top}\bigl(\flexmat^{\frac{1}{2}}\Psimat^{\frac{1}{2}}\flexmat^{\frac{1}{2}}\bigr)\utwoh^{-}.
\label{eq:right_boundary_final}
\end{align}
\endgroup

Combining \eqref{eq:left_boundary_final} and
\eqref{eq:right_boundary_final} with the right-minus-left orientation of the
boundary operator gives
\begin{equation}
\label{eq:total_boundary_contraction}
\externalboundary{\widehat{\Pimat\uh}-\tfrac12\Pimat\uh}{\uh}
=
(\uoneh^{+})^\top\bm S_0\uoneh^{+}\Big|_{x=0}
+
(\utwoh^{-})^\top\bm S_\ell\utwoh^{-}\Big|_{x=\ell}
=B(t)\ge0.
\end{equation}
Thus the characteristic boundary closure contributes nonnegative dissipation
to the semi-discrete energy identity.

\section{Tabulated Convergence Data}
\label{app:tabulated_convergence}

Tables~\ref{tab:mms_convergence_k1}--\ref{tab:mms_convergence_k3} list the average $L^2$ errors
underlying Figure~\ref{fig:l2_error_plots} together with the experimental orders of convergence,
$\mathrm{EOC}=\log_2\bigl(e(N_c)/e(2N_c)\bigr)$, computed from the tabulated values.

\begin{table}[htbp]
\centering\small
\caption{Convergence data for polynomial degree $k=1$ (average $L^2$ errors as in Figure~\ref{fig:l2_error_plots}; characteristic-upwind hyperbolic flux; EOC computed from the tabulated values).}
\label{tab:mms_convergence_k1}
\begin{tabular}{r cc cc @{\hspace{2em}} cc cc}
\toprule
 & \multicolumn{4}{c}{alternating LDG} & \multicolumn{4}{c}{central BR1} \\
\cmidrule(lr){2-5}\cmidrule(lr){6-9}
$N_c$ & \multicolumn{2}{c}{$\bm u_1$} & \multicolumn{2}{c}{$\bm u_2$} & \multicolumn{2}{c}{$\bm u_1$} & \multicolumn{2}{c}{$\bm u_2$} \\
\cmidrule(lr){2-3}\cmidrule(lr){4-5}\cmidrule(lr){6-7}\cmidrule(lr){8-9}
 & error & EOC & error & EOC & error & EOC & error & EOC \\
\midrule
4 & $9.19\times 10^{-2}$ & --- & $7.56\times 10^{-2}$ & --- & $9.51\times 10^{-2}$ & --- & $7.54\times 10^{-2}$ & --- \\
8 & $2.41\times 10^{-2}$ & 1.93 & $2.39\times 10^{-2}$ & 1.66 & $2.63\times 10^{-2}$ & 1.85 & $2.39\times 10^{-2}$ & 1.66 \\
16 & $5.45\times 10^{-3}$ & 2.15 & $6.84\times 10^{-3}$ & 1.80 & $6.30\times 10^{-3}$ & 2.06 & $6.87\times 10^{-3}$ & 1.80 \\
32 & $1.22\times 10^{-3}$ & 2.16 & $1.82\times 10^{-3}$ & 1.91 & $1.45\times 10^{-3}$ & 2.12 & $1.83\times 10^{-3}$ & 1.91 \\
64 & $3.20\times 10^{-4}$ & 1.93 & $4.70\times 10^{-4}$ & 1.95 & $3.43\times 10^{-4}$ & 2.08 & $4.72\times 10^{-4}$ & 1.96 \\
128 & $9.01\times 10^{-5}$ & 1.83 & $1.20\times 10^{-4}$ & 1.98 & $8.33\times 10^{-5}$ & 2.04 & $1.20\times 10^{-4}$ & 1.98 \\
\bottomrule
\end{tabular}
\end{table}

\begin{table}[htbp]
\centering\small
\caption{Convergence data for polynomial degree $k=2$ (average $L^2$ errors as in Figure~\ref{fig:l2_error_plots}; characteristic-upwind hyperbolic flux; EOC computed from the tabulated values).}
\label{tab:mms_convergence_k2}
\begin{tabular}{r cc cc @{\hspace{2em}} cc cc}
\toprule
 & \multicolumn{4}{c}{alternating LDG} & \multicolumn{4}{c}{central BR1} \\
\cmidrule(lr){2-5}\cmidrule(lr){6-9}
$N_c$ & \multicolumn{2}{c}{$\bm u_1$} & \multicolumn{2}{c}{$\bm u_2$} & \multicolumn{2}{c}{$\bm u_1$} & \multicolumn{2}{c}{$\bm u_2$} \\
\cmidrule(lr){2-3}\cmidrule(lr){4-5}\cmidrule(lr){6-7}\cmidrule(lr){8-9}
 & error & EOC & error & EOC & error & EOC & error & EOC \\
\midrule
4 & $3.49\times 10^{-3}$ & --- & $2.99\times 10^{-3}$ & --- & $3.85\times 10^{-3}$ & --- & $3.16\times 10^{-3}$ & --- \\
8 & $3.57\times 10^{-4}$ & 3.29 & $3.97\times 10^{-4}$ & 2.91 & $4.12\times 10^{-4}$ & 3.23 & $4.14\times 10^{-4}$ & 2.93 \\
16 & $2.72\times 10^{-5}$ & 3.71 & $4.08\times 10^{-5}$ & 3.28 & $3.11\times 10^{-5}$ & 3.73 & $3.41\times 10^{-5}$ & 3.60 \\
32 & $2.58\times 10^{-6}$ & 3.40 & $5.47\times 10^{-6}$ & 2.90 & $2.47\times 10^{-6}$ & 3.65 & $3.18\times 10^{-6}$ & 3.42 \\
64 & $3.37\times 10^{-7}$ & 2.94 & $7.86\times 10^{-7}$ & 2.80 & $2.71\times 10^{-7}$ & 3.19 & $3.71\times 10^{-7}$ & 3.10 \\
128 & $4.59\times 10^{-8}$ & 2.88 & $1.08\times 10^{-7}$ & 2.86 & $3.35\times 10^{-8}$ & 3.02 & $4.63\times 10^{-8}$ & 3.00 \\
\bottomrule
\end{tabular}
\end{table}

\begin{table}[htbp]
\centering\small
\caption{Convergence data for polynomial degree $k=3$ (average $L^2$ errors as in Figure~\ref{fig:l2_error_plots}; characteristic-upwind hyperbolic flux; EOC computed from the tabulated values).}
\label{tab:mms_convergence_k3}
\begin{tabular}{r cc cc @{\hspace{2em}} cc cc}
\toprule
 & \multicolumn{4}{c}{alternating LDG} & \multicolumn{4}{c}{central BR1} \\
\cmidrule(lr){2-5}\cmidrule(lr){6-9}
$N_c$ & \multicolumn{2}{c}{$\bm u_1$} & \multicolumn{2}{c}{$\bm u_2$} & \multicolumn{2}{c}{$\bm u_1$} & \multicolumn{2}{c}{$\bm u_2$} \\
\cmidrule(lr){2-3}\cmidrule(lr){4-5}\cmidrule(lr){6-7}\cmidrule(lr){8-9}
 & error & EOC & error & EOC & error & EOC & error & EOC \\
\midrule
4 & $6.32\times 10^{-5}$ & --- & $5.05\times 10^{-5}$ & --- & $6.73\times 10^{-5}$ & --- & $5.36\times 10^{-5}$ & --- \\
8 & $2.32\times 10^{-6}$ & 4.77 & $3.04\times 10^{-6}$ & 4.06 & $2.89\times 10^{-6}$ & 4.54 & $3.54\times 10^{-6}$ & 3.92 \\
16 & $8.08\times 10^{-8}$ & 4.84 & $1.43\times 10^{-7}$ & 4.41 & $1.17\times 10^{-7}$ & 4.63 & $2.09\times 10^{-7}$ & 4.09 \\
32 & $4.64\times 10^{-9}$ & 4.12 & $8.90\times 10^{-9}$ & 4.01 & $5.93\times 10^{-9}$ & 4.30 & $1.88\times 10^{-8}$ & 3.47 \\
64 & $3.13\times 10^{-10}$ & 3.89 & $5.74\times 10^{-10}$ & 3.95 & $3.43\times 10^{-10}$ & 4.11 & $2.02\times 10^{-9}$ & 3.21 \\
128 & $2.07\times 10^{-11}$ & 3.92 & $3.69\times 10^{-11}$ & 3.96 & $2.09\times 10^{-11}$ & 4.04 & $2.38\times 10^{-10}$ & 3.09 \\
\bottomrule
\end{tabular}
\end{table}

\FloatBarrier
\section*{Declaration of generative AI and AI-assisted technologies in the manuscript preparation process}
During the preparation of this manuscript, the authors used OpenAI ChatGPT for language editing and to improve the clarity and readability of the text. 
The authors reviewed and edited the output as needed and take full responsibility for the content of the manuscript.

\bibliographystyle{elsarticle-num}
\bibliography{references}

\end{document}